\documentclass{amsart}
\usepackage{amssymb,amsmath,amsthm,amsfonts,amscd,amstext}
\usepackage[mathscr]{eucal}
\usepackage{array}
\usepackage{graphicx}
\usepackage{cite}

\newtheorem{lem}{Lemma}[section]
\newtheorem{prop}[lem]{Proposition}
\newtheorem{thm}[lem]{Theorem}
\newtheorem{cor}[lem]{Corollary}

\newtheorem{sublem}[lem]{Sublemma}
\newtheorem{defn}{Definition}[section]
\newtheorem{rem}[lem]{Remark}

\begin{document}
\title{Induced map on $K$ theory
for certain $\Gamma-$equivariant maps between Hilbert spaces}
\author{  Tsuyoshi Kato}

\keywords
{$K$ theory, degree, 
 $\infty$ dimensional Bott periodicity}

\footnote
{ {\it Mathematics Subject Classification:}
    Primary 46L80, Secondary  46L85.}
    
\Large

\date{}

\begin{abstract}
Higson-Kapsparov-Trout 
introduced an infinite-dimensional Clifford algebra of a Hilbert space,
and verified Bott periodicity on $K$ theory.
To develop algebraic topology of maps between Hilbert spaces,
in this paper we introduce an induced Hilbert Clifford algebra,
and 
 construct an induced map between $K$-theory of the Higson-Kasparov-Trout Clifford algebra 
 and the induced Clifford algebra.  We also compute its $K$-group for some concrete case.
\end{abstract}

\maketitle

\section{Introduction}
Let $\Gamma $ be  a discrete group,
and 
$H, H'$ be two Hilbert  spaces on which $\Gamma$ acts linearly and  isometrically.
Let 
$F = l+c: H' \to H$ be a $\Gamma$-equivariant   map whose linear part is $l$,
which is also $\Gamma$-equivariant.
We want to construct is an ``induced map''
of $K$-theory of these  infinite-dimensional  spaces.
Of course we cannot obtain such a map in  the usual sense because
 these spaces are   locally non compact.
Thus, we introduce the infinite-dimensional
Clifford $C^*$-algebras  by Higson, Kasparov and Trout
 \cite{hkt}.

Let $E$ be a finite-dimensional Euclidean space, and let $Cl(E)$ be 
the complex Clifford algebra.
There is a $*$-homomorphism  
$\beta: C_0(\mathbb R) \to C_0(\mathbb R) \hat{\otimes} C_0(E, Cl(E))$
called the  Bott map, given by the functional calculus
$$f \to f(X \hat{\otimes} 1+ 1 \hat{\otimes } C)$$
where 
$X$ is an unbounded multiplier  of $C_0(\mathbb R)$
by $X(f)(x) = xf(x)$,
and $C$,
which is called the 
Clifford operator,
 is also  an unbounded multiplier  of $C_0(E,Cl(E))$ by
$C(v) =v$.
It turns out  that $\beta$ induces an isomorphism on $K$ theory as follows:
$$\beta_* : K_*(C_0(\mathbb R) )\cong
K_*( C_0(\mathbb R) \hat{\otimes} C_0(E,Cl(E))).$$
HKT generalized its construction
to obtain  the Clifford algebra
$S{\frak C}(H)$ for an   infinite-dimensional 
Hilbert space $H$, and verified the isomorphism
$$\beta_* : K_*(C_0(\mathbb R) )\cong
K_*( S{\frak C}(H)).$$
The idea is to use finite-dimensional approximation of the Hilbert space and  inductively
 apply the Bott map.

To develop algebraic topology of maps between Hilbert spaces,
our  first step is to construct an induced map in  $K$-theory.
Let $F = l+c : E' \to E$ be a proper map
such that $l: E \cong E'$ gives a linear isomorphism,
where $l$ is its linear part and $c$ is the non linear part
between finite-dimensional Euclidean spaces.
Then $F$ induces a map
$$F^* : C_0(E, Cl(E)) \to C_0(E',Cl(E'))$$
given by 
$$u \mapsto \Bigl ( v' \mapsto \bar{l}^* \bigl(u(F(v')) \bigr) \Bigr)$$
where $\bar{l}$ is the unitary of its polar decomposition.
Notice that 
the image $F^*(C_0(E, Cl(E)) ) \subset  C_0(E',Cl(E'))$
is a $C^*$ subalgebra.

It  becomes clear  why we  use $\bar{l}$
rather than $l$ to construct the infinite-dimensional version 
of this map. 
Let $F= l +c : H' \to H$ be a map between two Hilbert spaces.
To extend the above “pull-back” construction to the infinite-dimensional setting we have to impose extra conditions on $F$.
 We call such special maps finitely approximable. 
 See Definition \ref{fin-appro} in Section $3$ for more   detail. We then obtain the following result.

\begin{prop}
Suppose $F: H' \to H$ is finitely approximable.
Then there is an induced Clifford $C^*$ algebra
$S{\frak C}_F(H')$.
\end{prop}
This $C^*$ algebra
coincides with   $F^*(C_0(E, Cl(E)) )$ above,
in   finite-dimensional case.

If a discrete group $\Gamma$ acts linearly and isometrically,
then it also acts on  $S{\frak C}_F(H)$.

The following is our main theorem.
\begin{thm}
Suppose $F: H' \to H$ is finitely approximable.
Then it induces a $*$-homomorphism
$$F^*: S{\frak C}(H) \to S{\frak C}_F(H').$$
In particular it induces a homomorphism between $K$-groups
$$F^*: K_*(S{\frak C}(H)) \to K_*(S{\frak C}_F(H')).$$
\end{thm}
If a discrete group $\Gamma$ acts on both $H'$ and $H$ linearly and isometrically and $F$ is $\Gamma$-finitely approximable,
then $F^*$ is a $\Gamma$-equivariant $*$-homomorphism,
that induces a homomorphism between $K$-groups
$$F^*: K_*(S{\frak C}(H) \rtimes \Gamma) 
\to K_*(S{\frak C}_F(H')\rtimes \Gamma)$$
where the crossed product is full.

Suppose $F: H' \to H$ is strongly finitely approximable.
Then by approximating these Hilbert spaces
by finite-dimensional linear subspaces, we can 
obtain its degree $\deg(F) \in \mathbb Z$.
Then the above $F^*$ is given by
$$F^*:   \mathbb Z \to \mathbb Z$$
which sends $1$ to deg $(F)$  by choosing a suitable orientation.

We also compute 
the group 
$K(S{\frak C}_F(H')\rtimes \mathbb Z)$
for some concrete cases in Section $6$.

In a successive paper, we will apply our construction
of the $K$ theoretic induced map to a monopole map that appears 
in gauge theory.
Over a compact oriented four manifold, 
it turns our that the monopole map is strongly finitely approximable, 
and its degree coincides with the Bauer-Furuta degree when $b^1 =0$  \cite{bauer and furuta}.
We will verify that the covering monopole map 
on the universal covering space
is $\Gamma$-finitely approximable, when its linear part gives 
a linear isomorphism. This produces a higher degree map of Bauer-Furuta type.
The idea of the degree goes back to an old result by A.Schwarz \cite{schwarz}.

\section{ Infinite-dimensional Bott periodicity}

\subsection{Quick review of HKT construction}
We  review the construction 
of the Hilbert space Clifford $C^*$-algebras
 by Higson, Kasparov and Trout \cite{hkt}.

Let $E$ be a finite-dimensional Euclidean space, and let $Cl(E)$ be 
the complex Clifford algebras,
where we  choose  positive sign on the multiplication  $e^2 = |e|^2 1$
for every  $e \in E$.

  This  admits a natural ${\mathbb Z}_2$-grading.
The embedding $C: E \to Cl(E)$ gives a map which is called the Clifford operator.
Let us denote ${\frak C}(E)= C_0(E, Cl(E))$.
Let $X: C_0({\mathbb R}) \to C_0( {\mathbb R})$
be given by $X(f)(x)=xf(x)$.
Then $C_0({\mathbb R}) $ also admits  a natural ${\mathbb Z}_2$-grading
by even or odd functions.
Both operators $C$ and $X$ are degree one and essentially self-adjoint 
unbounded multipliers on ${\frak C}(E)$ and 
$C_0({\mathbb R})$ respectively.
In particular $X \hat{\otimes} 1+ 1 \hat{\otimes} C$ is a 
degree one and essentially self-adjoint 
unbounded multiplier on $C_0({\mathbb R})\hat{\otimes} {\frak C}(E)$.

Let us introduce 
a  $*$-homomorphism
$$\beta: C_0({\mathbb R}) \to S {\frak C}(E) :=
C_0({\mathbb R})\hat{\otimes} {\frak C}(E)$$
 defined by
$$\beta: f \to f(X \hat{\otimes} 1+ 1 \hat{\otimes} C)$$
through  functional calculus.
Let
$E$ be a separable real Hilbert space,  and
$E_a \subset E_b \subset E$ be a pair of finite-dimensional linear subspaces. We  denote
 the orthogonal complement  by $E_{ba} := E_b \cap E_a^{\perp}$.
Then we have  the canonical  isomorphism 
$S {\frak C}(E_b) \cong S{\frak C}(E_{ba}) \hat{\otimes}  {\frak C}(E_a)$
 of $C^*$ algebras.
Let us introduce   a $*$-homomorphism
passing through this isomorphism
$$\beta_{ba} = \beta \hat{\otimes} 1: S {\frak C}(E_a ) \to
S{\frak C}(E_{ba}) \hat{\otimes}  {\frak C}(E_a) = S{\frak C}(E_b) $$

\begin{lem}[HKT, Proposition $3.2$]
Let $E_a \subset E_b \subset E_c$.
Then the composition
\begin{gather*}
S {\frak C}(E_a) 
\;\stackrel{\beta_{ba}}{\longrightarrow} \;
S {\frak C}(E_b) 
  \stackrel{\beta_{cb}}{\longrightarrow} \; 
S {\frak C}(E_c) \;
\end{gather*}
coincides with the $*$-homomorphism
$$\beta_{ca} : S{\frak C}(E_a) \to S {\frak C}(E_c).$$
\end{lem}

\vspace{3mm}

\begin{defn}
We denote  the direct limit $C^*$-algebras by
$$S{\frak C}(E) = \varinjlim_a S{\frak C}(E_a)$$ 
where the direct limit is taken over all finite-dimensional linear
subspaces of $E$.
\end{defn}
It follows from the above construction that we can  obtain a 
 $*$ homomorphism
$$\beta: C_0({\mathbb R}) \to S{\frak C}(E).$$

Suppose a discrete group $\Gamma$ acts on $E$  linearly and  isometrically.
It induces  the action on $S{\frak C}(E)$ by
$$\gamma(f \hat{\otimes} u )(v) = f \hat{\otimes} \gamma(u(\gamma^{-1}(v))).$$
Thus,  the Bott map is $\Gamma$-equivariant.
For a $\Gamma$-$C^*$-algebra $A$, let us denote
$$K^{\Gamma}(A) := K(A \rtimes \Gamma)$$
where the right hand side $C^*$ algebra is given by the full crossed product of $A$ with $\Gamma$.

\begin{prop}[HKT, Theorem $3.5$]\label{HKT}
$\beta$ gives an equivariant  asymptotic equivalence from $S$ to $ S{\frak C}(E)$.

In particular it 
induces an isomorphism 
$$\beta_*: K_*^{\Gamma}(C_0({\mathbb R}))  \cong
 K_*^{\Gamma}(S {\frak C}(E)).$$
\end{prop}

\subsection{Direct limit $C^*$ algebras}
Let $H$ be a Hilbert space on which $\Gamma$ acts linearly and  isometrically.
Choose exhaustion by finite-dimensional linear subspaces $V_j \subset V_{j+1}$
 with dense union $\cup_j V_j \  \subset  H$.
Let $0< r_1 < \dots <r_i < r_{i+1} < \dots \to \infty$  be a divergent positive sequence with $r_{i+1} > \sqrt{2}r_i$, 
and let  $D_{r_i}^j \subset V_j$ be  the open disc with diameter $r_i$.

Consider the diagram of the embeddings
$$\begin{matrix}
& :  && : &&&& : \\
& \cap &&\cap &&&& \cap \\
\dots \subset & D_{r_i}^j & \subset & D_{r_{i+1}}^j  & \subset & \dots & \subset & V_j \\
& \cap && \cap &&&&  \cap \\
\dots \subset & D_{r_i}^{j+1} & \subset & D_{r_{i+1}}^{j+1}  & \subset & \dots & \subset & V_{j+1} \\
& \cap &&\cap &&& & \cap  \\
& : && : &&&& : \\
\dots \subset & D_{r_i} & \subset & D_{r_{i+1}} & \subset && \subset & H
\end{matrix}$$
Let $V_j^{\perp} \subset H$ be the orthogonal complement, and 
for $j' \geq j$,
 denote
\begin{align*}
&  V_{j,j'} := V_j^{\perp} \cap V_{j'},   \ \
  E_{r_i}^{j.j'} := D_{r_i}^j \times V_{j,j'}   \ \ \subset  \ \ V_{j'}, \\
 & E_{r_i}^j := D_{r_i}^j \times V_j^{\perp} \ \ \subset  \ \ H.
 \end{align*}

Let us set
$$S {\frak C}(D_{r_i}^j) \equiv C_0({\mathbb R}) \hat{\otimes} C_0(D_{r_i}^j ,Cl(V_j)). $$
Recall the Bott map
$$\beta: C_0({\mathbb R}) \to S{\frak C}(V) \qquad
f \to f(X \hat{\otimes} 1 + 1 \hat{\otimes} C)$$
for a finite-dimensional vector space $V$. Then 
we have $*$-homomorphisms
\begin{align*}
 \beta_{j,j'} = \beta \hat{\otimes} 1: 
 S {\frak C}(D_{r_i}^j)  &  \to 
  S C_0(V_{j,j'} ,Cl(V_{j,j'}))
  \hat{\otimes} {\frak C}(D_{r_i}^j)    \\
& \cong S {\frak C}(E_{r_i}^{j.j'})  \\
& \hookrightarrow  S {\frak C}(V_{j'})
 \end{align*}
where the last embedding is the open inclusion.

\begin{rem}
Trout developed 
a Thom isomorphism on  infinite-dimensional Euclidean bundles.
One may regard $C_0(D_{r_i}^j ,Cl(V_j)) $ 
as the set of continuous sections
on the Clifford algebra of the tangent bundle
$Cl(TD_{r_i}^j)$ vanishing at infinity.
Then $\beta_{j,j'}$ can be described as a $*$-homomorphism
$$\beta_{j,j'} = (1 \hat{\otimes} i_* ) 
\circ (\beta_{E^{j,j'}_{r_i}} \hat{\otimes}_{D^j_{r_i}}
\text{ id}_{D^j_{r_i}})$$
where $i : E^{j,j'}_{r_i} \hookrightarrow V'_j$ is the open inclusion.
See 
\cite{trout}  Section $2$.
\end{rem}

Let $S_r := C_0(-r,r) \subset C_0(\mathbb R)$,
and set
$$S_r {\frak C}(D^j_r) \equiv C_0(-r,r) \hat{\otimes} C_0(D_r^j, Cl(V_j)).$$
Then the above Bott map transforms as
\begin{align*}
 \beta_{j,j'} = \beta \hat{\otimes} 1: 
 S_{r_i} {\frak C}(D_{r_i}^j)    \to 
   S_{r_{i+1}} {\frak C}(D_{r_{i+1}}^{j'}).
 \end{align*}

\begin{lem}\label{cpt-supp}
The direct limit $C^*$-algebra
$$\lim_{i, j  \to \infty} S_{r_i} {\frak C}(D_{r_i}^j) = S{\frak C}(H)$$
coincides with the Clifford $C^*$-algebra of $H$.
\end{lem}
\begin{proof}
{\bf Step 1:}
We claim that the commutativity
$$\beta_{j,j''} = \beta_{j',j''} \circ
 \beta_{j,j'} 
$$
holds.
To make the notations clearer, 
let us denote   $\bar{\beta}_{j,j'}: S{\frak C}(V_{j}) \to  S{\frak C}(V_{j'}) $
as the standard Bott map. Then
the commutativity
$\bar{\beta}_{j,j''} =  \bar{\beta}_{j',j''} \circ
\bar{\beta}_{j,j'} 
$ holds.

For $a_{i,j} \in S_{r_{i}}{\frak C}(D_{r_{i}}^j)$, 
$\bar{\beta}_{j,j''}(a_{i,j}) = \beta_{j,j''}(a_{i,j})$ holds in $S{\frak C}(V_{j''})$,
 passing through
the isometric embedding $S_{r_{i+2}}{\frak C}(D_{r_{i+2}}^{j''})
\hookrightarrow S{\frak C}(V_{j''})$. 
This implies the equalities
\begin{align*}
\beta_{j,j''}(a_{i,j})  & = \bar{\beta}_{j,j''}(a_{i,j})  
=   \bar{\beta}_{j',j''}\circ
 \bar{\beta}_{j,j'}(a_{i,j})
 \\
& =   \bar{\beta}_{j',j''}\circ
\beta_{j,j'}
(a_{i,j})
= 
\beta_{j',j''}\circ
\beta_{j,j'}
(a_{i,j}).
\end{align*}
This commutativity allows us to construct the direct limit $C^*$-algebra
$$\lim_{i, j \to \infty} S_{r_i} {\frak C}(D_{r_i}^j) .$$
There is a canonical isometric embedding
$$I: \lim_{i, j  \to \infty} S_{r_i} {\frak C}(D_{r_i}^j)
\hookrightarrow  S{\frak C}(H).$$

{\bf Step 2:} 
It remains to verify that the image of $I$ is dense.
For a linear inclusion $V \hookrightarrow H$, let
$\beta: S{\frak C}(V) \to S{\frak C}(H)$ be the Bott $*$-homomorphism
 into the Clifford $C^*$ algebra.
An element $a \in S{\frak C}(H)$ is given as $\lim_j \beta(a_j)$ for some $a_j \in S{\frak C}(V_j)$.
Let $\chi_i \in C_c((-r_i,r_i); [0,1])$ and 
$\varphi_{i,j} \in C_c(D_{r_i}^j; [0,1])$ be cutoff functions
with $\chi_i|(-r_{i-1},r_{i-1}) \equiv 1$ and 
$\varphi_{i,j}|D_{r_{i-1}}^j \equiv 1$.
Let us set $\psi_{i,j} =\chi_i \hat{\otimes} \varphi_{i,j}$.
We claim that $b_{i,j} :=  \psi_{i,j} a_j \in S{\frak C}(D_{r_i}^j)$ converges to the same element:
$$\lim_{i,j \to \infty} \beta(b_{i,j}) = a \in S{\frak C}(H).$$
Choose any $j_0$ and $\epsilon >0$.
 There exists  $r >0$ such that $a_{j_0}$ 
satisfies the estimate
$||a_{j_0}||_{C_0((D_r)^c)} < \epsilon$.
For each $f \in C_0({\mathbb R})$, there is some $r >0$ such that
$\beta(f) \in  S{\frak C}(H) $
satisfies the estimate
$||\beta(f)||_{C_0((D_r)^c)} < \epsilon$.
Thus, any $a_j \in  S{\frak C}(V_j) $ with $j > j_0$
also satisfies the estimate
$$||a_j||_{C_0((D_{\sqrt{2}r})^c)}  \ < \ 2 \epsilon$$
for all large $r >>1$.
 This   verifies the claim.
\end{proof}
\vspace{3mm}

\subsection{Asymptotic unitary operators}\label{auo}
Let $H'$ be a Hilbert space.
For two finite-dimensional linear subspaces, let 
us set 
$$d'(V_1',V_2') = \sup_{v_2} \  \inf_{v_1}
 \ \{ \  ||v_1-v_2|| : ||v_1||= ||v_2||=1, \ v_i \in V_i' \ \}.$$
$d'(V_1'; V_2')=0$ holds if and only if  $V_1'$  contains $V_2'$. 
Then we  introduce the distance between these planes by
$$d(V_1',V_2') = \min \ \{\  d'(V_1'; V_2') ,\  d'(V_2' ; V_1') \ \}.$$

Let $l: H' \to H$ be a linear isomorphism between Hilbert spaces,
and let
$$\bar{l} = l \circ \sqrt{l^* \circ l}^{-1}   : H' \to H$$ be the unitary corresponding to the polar decomposition of $l$.
For any finite-dimensional linear subspace $V \subset H$, 
let us compare two linear subspaces
 $$V' \equiv l^{-1}(V), \quad \bar{V}' \equiv \bar{l}^{-1}(V) 
 \quad \subset \quad 
H'.$$

\vspace{3mm}

The following lemma will not be used later,
but may be useful to understand how $V'$ and $\bar{V}'$ differ from each other.

\begin{lem}
Let $W'_i $ be 
 a family of  finite-dimensional linear subspaces
 with  $W'_i \subset W_{i+1}'$ so that 
 the union   $\cup_i W'_i \subset H'$ is dense.

For  any finite-dimensional linear subspace
$V' \subset H'$ and any small $\epsilon >0$, 
there is some $i_0$ such that for all $i \geq i_0$, 
 $$||(1- \bar{\text{pr}_i})|\bar{V}' || < \epsilon$$
holds, where 
$\bar{\text{pr}}_i:  H' \to \bar{W}_i'$ is the orthogonal projection
and  $\bar{W}'_i := \bar{l}^{-1}(l(W_i'))$.
\end{lem}
\begin{proof}
It is sufficient  to verify that for 
any finite-dimensional linear subspace $V' \subset H'$
and any $\epsilon >0$,  the estimate
$$d(\bar{V}' , \bar{W}'_i) < \epsilon$$
holds for all large $i >>1$.
Actually $\bar{W}' =H'$ holds when $W'=H'$
since  the polar decomposition gives the unitary.
Thus,  for any finite-dimensional linear exhaustion
$W'_i $ such that  $\cup_i W'_i \subset H'$ is dense, 
$\cup_i \bar{W}'_i \subset H'$ is also dense.
Therefore the estimate  holds.
\end{proof}


\begin{defn}\label{as-unitary}
Let  $H'$ and $H$   be two Hilbert spaces and
$l: H' \to H$ be a linear isomorphism.

$l$ is asymptotically unitary if for any $\epsilon >0$,
there is a finite-dimensional linear subspace 
$V \subset H'$ such that
the restriction
$$l: V^{\perp} \cong l(V^{\perp})$$
satisfies the estimate on its operator norm
$$||(l-\bar{l})| V^{\perp}|| < \epsilon$$
 where $\bar{l}$ is the unitary of the polar decomposition of
  $l: H' \to H$.
\end{defn}

\begin{rem}
In  a subsequent paper, we will verify that 
a self-adjoint elliptic operator on a compact manifold 
is asymptotically unitary between Sobolev spaces.
\end{rem}

\begin{lem}\label{dist}
Let $l: H' \cong H$ be asymptotically unitary.
For any $\epsilon >0$, there is a finite-dimensional vector subspace
$V_0' \subset H'$ such that the estimate
$$d(V', (\bar{l}^* \circ  l)(V')) < \epsilon$$
holds for any $V' \supset V'_0$.
\end{lem}

\begin{proof}
Let $V' \subset H'$  be a closed linear subspace, and $(V')^{\perp}$
be  its orthogonal complement.  
Let $\text{pr}: H' \to V'$ be the orthogonal projection.

{\bf Step 1:}
Take a finite-dimensional subspace $V_0' \subset H'$
so that 
  $l$  
 satisfies the estimate
 $||(l-\bar{l})| (V')^{\perp}|| < \epsilon$ for any $V' \supset V_0'$.
 Then   the operator norm of the restriction satisfies the estimate
$$||(\bar{l}^*l - \text{pr} \circ \bar{l}^*l )|V'|| < \epsilon.$$
 Decompose the operator $\bar{l}^* l$ with respect to 
$V' \oplus (V')^{\perp}$, and express 
$\bar{l}^* l $ by a matrix form
$$\begin{pmatrix}
 A &B \\
 C & D
\end{pmatrix}$$
where both the estimates
$$||D - \text{id }|| , ||B|| \ < \epsilon$$
hold.

{\bf Step 2:}
$C =B^*$ holds since $\bar{l}^*l $ is self-adjoint.
Hence  the estimate $||C|| < \epsilon$ also holds.
Then the conclusion holds because  the estimate 
$$d(V', (\bar{l}^* \circ  l)(V')) \leq ||C||$$
holds.
\end{proof}

\subsection{A variant of  Clifford $C^*$-algebra}\label{var-bott}
Let us introduce a variant of the  HKT construction.
Ultimately,  the result of the $C^*$-algebra turns out to be
$*$-isomorphic to  the original one given
 by HKT. This variant  naturally appears when one considers 
 the induced Clifford $C^*$-algebra we introduce later.

Let $l: H' \cong H$ be an asymptotically unitary  isomorphism.
Let $E \subset H$ be a finite-dimensional Euclidean space, and
denote
$$E' = l^{-1}(E) , \quad \bar{E}' = \bar{l}^{-1}(E).$$
The map
$$C_l \equiv \bar{l}^{*} \circ l: E' \to  \bar{E}' \hookrightarrow Cl(\bar{E}')$$ 
 is  called the {\em induced Clifford operator}.
Let us denote
 $${\frak C}_l(E')= C_0(E', Cl(\bar{E}'))$$
  and
introduce a  $*$-homomorphism
$$\beta_l: C_0({\mathbb R}) \to S {\frak C}_l(E') \equiv 
C_0({\mathbb R})\hat{\otimes} {\frak C}_l(E')$$
 defined by
$\beta_l: f \to f(X \hat{\otimes} 1+ 1 \hat{\otimes} C_l)$
by functional calculus.

Let
$E'_a \subset E'_b \subset H'$ be a pair of finite-dimensional linear subspaces, and denote
 the orthogonal complement  as $E'_{ba} = E'_b \cap (E'_a)^{\perp}$.

\begin{lem}\label{l-bott}
Let $l:H' \cong H$ be asymptotically unitary.

Then 
 there is a finite-dimensional vector space $V' \subset H'$ such that
there is a canonical $*$-isomorphism
$$I_{ba}: S{\frak C}_l(E'_{ba}) \hat{\otimes}  {\frak C}_l(E'_a) \cong
S {\frak C}_l(E'_b) $$
 if the inclusion $V' \subset E'_a$ holds.
 \end{lem}
 
 \begin{proof}
 Let $V'$ be the vector subspace  given by Lemma \ref{dist}. 
 Let $\hat{E}'_{ba} \subset \bar{E}'_b$ be the orthogonal complement of $\bar{E}'_a$,
 and consider the orthogonal projection
 $$\hat{\text{pr}}: \bar{E}'_{ba} \to \hat{E}'_{ba}.$$
 
 By the assumption,
  $l^* \circ l$ is almost unitary on $E'_{ba}$
 so that the operator norm satisfies the estimate
 $||(l^* \circ l ) - \text{ id } | E'_{ba}|| < \epsilon$.
 The estimate
 $d(E_a', \bar{E}_a') < \epsilon$ also holds  from Lemma \ref{dist}. 
Thus,  the operator norm of the above projection
 satisfies the estimate
 $$||\hat{\text{pr}} - \text{ id } |  \bar{E}'_{ba}|| < 2 \epsilon.$$
 
 In particular the projection gives an isomorphism.
 Let $\hat{\text{\bf pr}}: \bar{E}'_{ba} \to \hat{E}'_{ba}$
 be the unitary of the polar decomposition.
 It also satisfies the estimate
  $||\hat{\text{\bf pr}} - \text{ id } |  \bar{E}'_{ba}|| < 4 \epsilon$,
which  induces a $*$-isomorphism
 $$\hat{\text{\bf pr}}: Cl(\bar{E}'_{ba}  )  \cong  Cl( \hat{E}'_{ba}).$$
 It extends to the 
 $*$-isomorphism
 $$\hat{\text{\bf pr}} \hat{\otimes} 1: Cl(\bar{E}'_{ba}  )  \hat{\otimes}
 Cl(\bar{E}'_{a}  ) 
  \cong  Cl( \hat{E}'_{ba})\hat{\otimes}
 Cl(\bar{E}'_{a}  )  \cong 
 Cl(\bar{E}'_{b}  ) $$
  which induces the desired $*$-isomorphism
  $$S{\frak C}_l(E'_{ba}) \hat{\otimes}  {\frak C}_l(E'_a) \cong
S {\frak C}_l(E'_b) .$$
 \end{proof}

 It follows from Lemma \ref{l-bott} that there
  is a canonical $*$-homomorphism
$$\beta_{ba} = \beta_l \hat{\otimes} 1: S {\frak C}_l(E'_a ) \to
S{\frak C}_l(E'_{ba}) \hat{\otimes}  {\frak C}_l(E'_a) \cong  S{\frak C}_l(E'_b) .$$

 \begin{rem}
 Surely  $\hat{\text{\bf pr}}$ induces a linear map
 $$\hat{\text{\bf pr}}: Cl(\bar{E}'_{ba}  )  \to Cl( \hat{E}'_{ba})$$
 by setting $u =u_1+u_2 \in Cl(\hat{E}'_{ba}) \oplus Cl^0(\bar{E}'_{a})$ to
 $u_1 \in Cl(\hat{E}'_{ba}) $, where $Cl^0(E)$ is the scalarless part of $Cl(E)$.
 However it cannot be ``almost'' $*$-isomorphic in general,
 as $\dim E'_{ba}$ grows.
 To see this,  let us 
take any $u' \in \bar{E}'_{ba}$ 
and set $u'''  = u' - \hat{\text{\bf pr}}(u')  := u' - u''$.
For any orthonormal basis $\{u'_1, u'_2,  \dots \}$ of $\bar{E}'_{ba}$,
consider their product
$u'_1 u'_2 \dots u_m' \in Cl(\bar{E}'_{ba})$.
 \begin{align*}
  u'_1 u'_2 \dots & = (u''_1+u_1''') ( u''_2 +u_2''')(u_3''+u_3''')  \dots  (u_m''+u_m''') \\
  & = u''_1 u''_2 \dots  u_m'' + \text{ other terms }
  \end{align*}
  Each norm $||u_i''|| < 1$ is strictly less than $1$, and hence, the norm of 
  their product in the first term above may degenerate to zero.
 \end{rem}

Let $A_a$ be a family of $C^*$-algebras, and $\beta_{ba}: A_a \to A_b$ be 
a family of $*$-homomorphisms, where $\{a\}$ is a semi ordered set.
The family  $\{\beta_{ba}\}_{b,a}$ is {\em asymptotically commutative}
if for any $\epsilon >0$, there is $a_0$ such that for any triplet
 $c \geq b \geq a \geq a_0$,
the estimate
$$||\beta_{ca} - \beta_{cb} \circ \beta_{ba}|| < \epsilon$$
holds.

For $v_a \in A_a$, introduce   the set of equivalent classes
$$\bar{v}_a := \{ \beta_{ba}(v_a)\}_{b \geq a}$$
divided by all elements $\bar{v}'_a$ with
$$\lim_{b} \ ||\beta_{ba}(v'_a)||=0.$$
Consider the algebra generated by elements of the form
$\bar{v}_a$, where the sum is given by
$$\bar{v}_a +\bar{v}_{a'} = \{ \beta_{ba}(v_a )+ \beta_{ba'}(v_{a'}) \}_{b \geq a,a'}$$
 and the multiplication is given by
$$\bar{v}_a \cdot \bar{v}_{a'} = 
\{ \beta_{ba}(v_a) \cdot \beta_{ba'}(v_{a'}) \}_{b \geq a,a'}.$$

The direct limit $C^*$-algebra $A$
with respect to the family $\{\beta_{ba}: A_a \to A_b\}_{a,b}$
 is defined by
the closure of the above algebra with the norm
$$||\bar{v}_a||: = \lim_{b} \ ||\beta_{ba}(v_a)|| \qquad \qquad  (*)$$
We also denote it as
$$A := \varinjlim_a \ A_a.$$

Let us set
$$\beta_{ba} : = I_{ba} \circ \beta_l: 
S {\frak C}_l(E_a')  \to S {\frak C}_l(E_{ba}') \hat{\otimes} {\frak C}_l(E_a') 
\cong S {\frak C}_l(E_b')$$ 
where $E_a'$ run over all finite-dimensional subspaces, and 
$b \geq a$ if and only if $E_b' \supset E_a'$ holds.

It follows from the proof of Lemma \ref{l-bott} that the following lemma holds.

\begin{lem}
The family $\{ \beta_{ba}  : S {\frak C}_l(E_a') \to  S {\frak C}_l(E_b')\}_{b,a}$
is asymptotically commutative.
\end{lem}

\begin{defn} \label{SC_l}
Let $l: H' \cong H$ be asymptotically unitary.
The direct limit $C^*$-algebra is given by
$$S{\frak C}_l(H') = \varinjlim_a \ S{\frak C}_l(E'_a)$$
where the norm is given in $(*)$    above.
\end{defn}

\begin{prop}\label{iso}
Assume $l$ is  asymptotically unitary.

Then there is a canonical $*$-isomorphism
$$S {\frak C}_l(H') \to S {\frak C}(H')$$
between the two Clifford $C^*$ algebras.

If a group $\Gamma$ acts on $H'$ linearly 
and  isometrically and $l$  is $\Gamma$-equivariant,
then this  $*$-isomorphism 
 is $\Gamma$-equivariant.
\end{prop}
\begin{proof}
{\bf Step 1:}
It follows from Lemma \ref{dist} and the assumption
 that
for any $\epsilon >0$, there is a finite-dimensional vector space $V'_0 \subset H'$
such that for all $E_a' \supset V'_0$, the following two  estimates  hold:
\begin{align*}
& d(E_a', \bar{l}^* \circ l(E_a')) \ < \ \epsilon, \\
& d((E_a')^{\perp}, \bar{l}^* \circ l((E_a')^{\perp})) \ < \ \epsilon.
\end{align*}
 
Take another $E_b' \supset E_a'$ with $E'_{ba}$, and let
 $\text{pr}_1: \bar{E}_a' \cong E_a'$ 
and  $\text{pr}_2:  \bar{E}_{ba}' \cong E_{ba}'$
be the orthogonal projections.
Their corresponding unitaries
 $\text{\bf pr}_i$ 
 satisfy the bounds
$$||\text{\bf pr}_i - \text{ id }|| < 2 \epsilon.$$

They extend to   $*$-isomorphisms
\begin{align*}
&  \text{\bf pr}_1: Cl(\bar{E}_a') \cong Cl(E_a') , \\
&  \text{\bf pr}_2:  Cl(\bar{E}_{ba}' ) \cong Cl(E_{ba}').
\end{align*}
In particular they induce the  $*$-isomorphisms
\begin{align*}
& \text{\bf pr}_1: C_0(E_a', Cl(\bar{E}_a') )\cong C_0(E_a', Cl(E_a')), \\
&  \text{\bf pr}_2: C_0(E_{ba}', Cl(\bar{E}_{ba}') )\cong C_0(E_{ba}', Cl(E_{ba}')).
\end{align*}

{\bf Step 2:}
Let us consider two Bott maps
\begin{align*}
& \beta_1: C_0({\mathbb R}) \to S {\frak C}_l(W') , \quad 
\beta_1( f)  = f(X \hat{\otimes} 1+ 1 \hat{\otimes} C_l), \\
& \beta_2: C_0({\mathbb R}) \to S {\frak C}(W') ,
\quad
\beta_2( f)  = f(X \hat{\otimes} 1+ 1 \hat{\otimes} C)
\end{align*}
and  the diagram
$$\begin{CD}
S {\frak C}_l(E_a')  @> \beta_1 >> S {\frak C}_l(E_{ba}') \hat{\otimes}  
C_0(E_a', Cl(\bar{E}_a')) \\
@V 1 \hat{\otimes} \text{\bf pr}_1 VV 
@V 1 \hat{\otimes} \text{\bf pr}_2  \hat{\otimes} \text{\bf pr}_1 
  VV \\
S {\frak C}(E_a')  @> \beta_2 >> S {\frak C}(E_{ba}') \hat{\otimes} C_0(E_a', Cl(E_a') )
\end{CD}$$
Denote 
$\text{\bf pr}_{21}: =  
\text{\bf pr}_2  \hat{\otimes} \text{\bf pr}_1
$. 
Then  this diagram satisfies the estimate
$$||1 \hat{\otimes} \text{\bf pr}_{21} \circ \beta_1
\  - \  \beta_2 \circ 1 \hat{\otimes} \text{\bf pr}_1|| \ < \ 4 \epsilon.$$ 
$\epsilon$ can be arbitrarily small by choosing large $E_a'$.

{\bf Step 3:}
Let us take an element $x \in S{\frak C}_l(H') $, and choose $x_a \in S{\frak C}_l(E'_a) $
with $\lim_a \ ||\beta_1(x_a) -x||  =0$,
 where $\beta_1(x_a) \in S{\frak C}_l(H')$.
It follows from the above estimate on the diagram that 
\begin{align*}
& \text{\bf pr}: S{\frak C}_l(H') \to S{\frak C}(H') , \\
& \qquad \text{\bf pr} (x) =  \lim_a \ \beta_2( 1 \hat{\otimes}\text{\bf pr}_1 (x_a)) 
\end{align*}
 is uniquely defined and independent of choice of $x_a$.
 
 It is easy to check that this assignment gives a $*$-homomorphism.
 To see that it is isomorphic, we consider a  converse projection,
 from $\bar{\text{pr}}' : E_a'  \cong \bar{E}_a'$.
 A parallel argument gives another $*$-homomorphism
 $\text{\bf pr}' : S{\frak C}(H') \to S{\frak C}_l(H') $,
 and their compositions give the required   identities.
 
 {\bf Step 4:}
 Let us consider $\Gamma$-equivariance.
 Suppose $\Gamma$ acts on $H'$ linearly and  isometrically.
 We claim that 
 $\text{\bf pr}_1:  \bar{E}_a'  \to E_a'$ is $\Gamma$-equivariant.
 
 To see this, notice that $\bar{l} $ and hence $\bar{l}^* \circ l$ are 
 both $\Gamma$-equivariant.  Then we have the equalities
 $$\gamma( \bar{E}_a') = \gamma(\bar{l}^* \circ l(E_a')) = \bar{l}^* \circ l(\gamma E_a')) = \overline{\gamma E_a'}.$$
 Therefore,
  $$ \text{\bf pr}_1( \gamma(\bar{E}_a'))
 = \text{\bf pr}_1( \overline{\gamma E_a'}) =\gamma(E_a')
 = \gamma (\text{\bf pr}_1( \bar{E}_a')).$$
 As the Bott map is also $\Gamma$-equivariant, 
 the process from step $1$ to step $3$ works
 equivariantly.
 \end{proof}

\section{Finite-dimensional approximation}\label{fin-dim-appr}
Let $F : H' \to H$ be a metrically proper map between Hilbert spaces.
Then, there is a proper and increasing 
 function $g :[0, \infty) \to [0, \infty)$  
such that the lower bound 
$$g(||F(m) ||  ) \geq  ||m||$$
holds for all $m \in H'$.  
 Later we analyze a family of maps
 of the form $F_i : B_i' \to W_i$, where 
 $W_i \subset H$ is a finite-dimensional linear subspace and 
 $B_i' \subset W_i' \subset H'$ is 
 a closed and bounded set in a finite-dimensional linear space.

 Let $D_t \subset H$ be a $t$ ball.
 We say that the family of maps
$ \{F_i\}_i$
  is {\em proper}, 
if  there are positive and increasing numbers 
 $r_i,s_i \to \infty$ such that the inclusion holds:
 $$F_i^{-1}(D_{s_i} \cap W_i)  \ \subset \   D_{r_i} \cap W_i' .$$

Denote  $F = l+c$ where $l$ is its linear part and $c$ 
is a non linear  term.

\begin{lem}\label{met-pro}
Let $F = l+c: H' \to H$ be a metrically proper map.
 Suppose  $l$ is surjective and $c$ is compact on each bounded set.
 Then there is 
 a proper and increasing function 
 $f:[0, \infty) \to [0, \infty)$  
such that
the following holds:
  for any $r >0$ and $1 \geq \delta_0 >0$,
   there is a finite-dimensional linear subspace $W'_0 \subset H'$
 such that for any linear subspace $W_0' \subset W' \subset H'$,
 the composed map
 $$\text{ pr } \circ F: D_r  \cap W' \to  W$$
 also satisfies the bound
 $$ f(||\text{ pr } \circ F(m) || ) \geq  ||m||  $$
 for any $m \in D_r \cap W'$, where $W= l(W')$ and $\text{pr}$ is 
 the orthogonal projection to $W$.
 
 Moreover the estimate holds
 $$\sup_{m \in D_r \cap W'} 
 ||F (m)- \text{ pr } \circ F(m)|| \  \leq  \ \delta_0.$$
   \end{lem}
 \begin{proof}
 Let $C  \subset H$ be the closure of the image 
 $c(D_r)$, which is compact. Hence there are finitely many 
  points 
 $w_1, \dots, w_k \in c(D_r)$ such that their $\delta_0$ neighborhoods cover
 $C$. 
 
 Choose $w_i' \in H'$ with  $l(w_i') =w_i$  for $1 \leq  i \leq k$,
 and let $W'_0$ be the linear span of these $w_i'$.

 The restriction
  $\text{ pr } \circ F: D_r  \cap W'_0 \to  W_0$
  satisfies the equality
  $$\text{ pr } \circ F = l +  \text{ pr } \circ c$$
  where $W_0 =l(W_0')$.  
  Then for any $m \in D_r \cap W'_0$, there is some $w_i' $ with $||c(m) -c(w_i')|| \leq \delta_0$, and
  the estimate $||F(m) - \text{ pr } \circ F(m)|| \leq \delta_0$
  holds.
  
  Since $g$ is increasing, we obtain the estimates
  $$ g(||\text{ pr } \circ F(m) ||  +\delta_0) \geq 
   g(|| F(m) || ) \geq  ||m||  .$$ 
   The function $f(x) = g(x+1)$ satisfies the desired property.
   
   For any other linear subspace $W_0' \subset W' \subset H'$, 
   the same property holds for  $\text{ pr } \circ F: D_r  \cap W' \to  W$ with $W =l(W')$.
   \end{proof}

Let $W_i' \subset H'$ and $W_i \subset  H$ be two families of  
finite-dimensional linear subspaces.
Let us say that a family of linear isomorphisms
$$l_i : W_i' \cong W_i$$
is an {\em asymptotic unitary  family} if 
the following conditions hold:

\vspace{2mm}

$(1)$ there exists an asymptotically unitary map $l: H' \cong H$, 

\vspace{2mm}

$(2) $ for each $i_0$,
$\lim_i \ || l-l_i||_{W'_{i_0}}=0$
holds,  where  $l,l_i: W'_{i_0} \to H$, and

\vspace{2mm}

$(3)$  uniform bounds 
$C^{-1} || l || \leq  ||l_i|| \leq C || l ||$ hold on their norms, 
where $C$ is independent of $i$.

\vspace{2mm}

Let us introduce an approximation of $F$ as a family of maps on finite-dimensional 
linear subspaces.
Let $D_{r_i}' \subset H'$ and $D_{s_i} \subset H$
be   $r_i$ and $s_i$ balls respectively.

\begin{defn}\label{fin-appro}
Let $F = l+c: H' \to H$ be a metrically proper map,
where $l$ is its linear part
and $c$ is   a nonlinear term.
Let us say that $F$ is  finitely approximable if there is
an increasing family of finite-dimensional linear subspaces
$$ W'_0   \subset   W'_1 
 \subset   \dots  \subset  W'_i   \subset  \dots  \subset H' $$ 
 and a family of maps
 $F_i = l_i +c_i : W_i' \to W_i$, where $W_i  =l_i(W_i')$,
 such that

$(1)$ the  union 
$\cup_{i \geq 0} W'_i \ \subset H'$ is dense, 

$(2)$ 
 there are two  sequences
 $s_0 <s_1<  \dots \to \infty$ and
 $r_0 <r_1< \dots \to \infty$ 
with $r_i \geq s_i$
 such that
 the embedding
$$F_i^{-1}(D_{s_i} \cap W_i ) \ \subset \ D_{r_i}'  \cap W_i'$$
holds
for all $i$,

$(3)$ for each $i_0$, 
$$\lim_{i \to \infty} \  
\sup_{m \in D_{r_{i_0}}' \cap W_{i_0}'} \ ||F(m) - F_i(m)|| =0,$$

$(4)$ 
$l_i: W_i' \cong W_i$ is an asymptotic unitary family
with respect to $l$.
\end{defn}
Let us also say that $F$ is {\em strongly finitely approximable}  if
it is finitely approximable,
$l_i = l |_{W_i'}$  and $c_i = \text{pr}_i \circ c$, such that
$$\lim_{i \to \infty} \ ||(1 - \text{pr}_i ) \circ c| D_{r_i}' || =0$$
 where 
$\text{pr}_i : H \to W_i$ is the orthogonal projection.

\vspace{3mm}

The following restates Lemma \ref{met-pro}.

\begin{cor}\label{met-prop.cpt}
Let $F = l+c: H' \to H$ be a metrically proper map such that 
   $l$ is asymptotically unitary and $c$ is compact on each bounded set.
Then $F$ is strongly finitely approximable.
\end{cor}

\vspace{3mm}

Suppose both $H'$ and $H$ admit linear isometric  
actions by a group  $\Gamma$ and 
assume that both $F$ and $l$ are $\Gamma$-equivariant where $F = l+c$.
Then we say that 
$F$ is  $\Gamma$-{\em  finitely approximable}, if moreover 
the above family $\{W'_i\}_i$  satisfies that  the union
 $$ \cup_i \  \{ \ \gamma(W_i') \cap W_i'  \ \} \ \subset \ H'$$
 is dense
for any $\gamma \in \Gamma$.

Note that 
the above family $\{F_i\}_i$  satisfies
 convergence  for any  $\gamma \in \Gamma$
$$\lim_{i \to \infty}  \  \sup_{m 
}  \ ||\gamma F_i( m) - F_i( \gamma m)|| =0$$
where $m \in D_{r_{i_0}}' \cap
 W'_{r_{i_0}} \cap \gamma^{-1}(W'_{r_{i_0}})$
because  the estimate
\begin{align*}
 ||\gamma F_i( m) - F_i( \gamma m)|| & \leq
  ||\gamma F( m) - \gamma F_i( m)||+
   ||\gamma F( m) - F_i( \gamma m)|| \\
   & =  ||F( m) - F_i( m)||+
   ||F( \gamma m) - F_i( \gamma m)||
   \end{align*}
holds.

Let us take $\gamma \in \Gamma$, 
and consider the $\gamma$ shift of the finite approximation data
$$\gamma(W_i'), \quad
\gamma^*(F_i), \quad  \gamma^*(l_i).$$
It is clear that the above shift 
gives another finite approximation of $F$.

\vspace{3mm}

\section{Induced Clifford $C^*$-algebra}
Let $F =l+c: H' \to H$ be a map. 
We aim here is to construct an ``induced'' 
Clifford $C^*$-algebra $S{\frak C}_F(H')$.

\subsection{Model case}\label{Model}
Let us start with a model case  that consists of 
a proper and nonlinear map
$$F =l +c: E' \to E$$ 
between finite-dimensional Euclidean spaces,
where $l$ is a linear isomorphism. 
Consider a  $*$-homomorphism
$$F^*: S {\frak C}(E)
 \to S {\frak C}(E')  =  C_0(\mathbb{R}) \hat{\otimes}
 C_0(E', Cl(E'))$$
 defined by
$F^*(  f \hat{\otimes} u)(v') :=  f \hat{\otimes} \bar{l}^{-1}(u(F(v'))$,
and denote its image by
$$S {\frak C}_F(E')  = F^*(S {\frak C}(E))$$
which is a $C^*$-subalgebra  in $S {\frak C}(E')$,
whose norm 
is denoted  by $||\quad ||_{S {\frak C}_F}$.

The induced map
$$C_F \equiv \bar{l}^{-1} \circ F: E' \to E'  \hookrightarrow Cl(E')$$ 
 is   called the {\em  induced Clifford operator}. 
  We use it to
introduce a  $*$-homomorphism
$$\beta_F: C_0({\mathbb R}) \to S {\frak C}_F(E') $$
 defined by
$\beta_F: f \to f(X \hat{\otimes} 1+ 1 \hat{\otimes} C_F)$
by functional calculus.

Now suppose a Hilbert space $H'$ is spanned by 
an infinite family of finite-dimensional Euclidean planes as
$$E'_1 \oplus E'_2 \oplus \dots$$
and assume  there is a family of proper maps
$$F_i = l_i +c_i : E'_i \to E_i$$
which extends to a map
$$F =(F_1,F_2, \dots)  = l+ c: H' \to H$$
where $H$ is spanned by $E_1\oplus E_2 \oplus \dots$.
Assume  $l =(l_1,l_2, \dots): H' \cong H$ 
is asymptotically unitary.

\begin{lem}
Let $F =(F_1,F_2)$ be diagonal as above. Then
\[
 {\frak C}_F
(E'_1\oplus E'_2 ) 
\cong 
{\frak C}_{F_1}
(E'_1) \hat{\otimes}
 {\frak C}_{F_2} 
(E'_2).
\]
\end{lem}

\begin{proof}
By definition
$
 {\frak C}_{F_i}
(E'_i) 
=F_i^*( {\frak C}
(E'_i))
$
holds for $i =1,2$.
Hence we have the isomorphisms

\begin{align*}
 {\frak C}_F 
(E'_1\oplus E'_2 ) 
&  \cong 
F^* ( {\frak C}(E'_1\oplus E'_2 ) ) 
  \cong 
(F_1 \hat{\otimes} F_2)^*
 {\frak C}
(E'_1) \hat{\otimes}
{\frak C}
(E'_2)  \\
& \cong
F_1^*( {\frak C}
(E'_1)) \hat{\otimes}
F_2^*( {\frak C}
(E'_2) ) \\
& \cong 
  {\frak C}_{F_1}
(E'_1) \hat{\otimes}
 {\frak C}_{F_2} 
(E'_2).
\end{align*}
\end{proof}

 Then the induced Bott map is given by
\begin{align*}
\beta_{F_{i+1}}  & : S {\frak C}_F
(E'_1\oplus \dots \oplus E'_i ) \to \\
& S{\frak C}_F(E'_{i+1}) \hat{\otimes}  {\frak C}_F
(E'_1\oplus \dots \oplus E'_i) 
\cong  S{\frak C}_F(E'_1 \oplus  \dots \oplus E'_{i+1}) 
\end{align*}
by use of 
 $C_{F_{i+1}}$.

  More generally, one can induce
 $$
\beta_{i, j}   : S {\frak C}_F
(E'_1\oplus \dots \oplus E'_i ) \to 
  S{\frak C}_F(E'_1 \oplus  \dots \oplus E'_{j}) 
$$
by use of  a canonical extension
 $$C_{(F_{i+1} , \dots , F_j)}:E'_{i+1} \oplus \dots \oplus E'_j \to 
 E'_{i+1}  \dots \oplus E'_j \subset Cl(E'_{i+1} \dots \oplus E'_j).$$
 Let $u \in S {\frak C}_F(E'_1\oplus \dots \oplus E'_i ) $ for some $i$.
Then the limit
$$||u|| \equiv \lim_{j \to \infty}
 ||\beta_{i,j}(u)||_{S {\frak C}_F}$$
exists, which   gives a norm on 
$S {\frak C}_F(E_1' \oplus E_2' \oplus \dots)$.
 Then the  direct limit $C^*$ algebra is given by
$$S{\frak C}_F(H') = \lim_j \  S {\frak C}_F(E'_1\oplus \dots \oplus E'_j )
$$
whose norm is given as above.

Notice that $S{\frak C}_F(H')$ is no longer a $C^*$-subalgebra of $S{\frak C}(H')$.

\begin{lem}\label{F=l}
In the case when
   $c_i \equiv 0$ and hence  $F_i =l_i$
    for all $i$,
  the induced   Clifford $C^*$-algebra admits a
canonical $*$-isomorphism
$$S{\frak C}_F(H')  \ \cong \ S{\frak C}(H').$$
\end{lem}

\begin{proof}
This 
follows from Proposition \ref{iso} 
with the coincidence
$$S{\frak C}_F(H')  \ = \ S{\frak C}_l(H')$$
where the right-hand side is given in Definition \ref{SC_l}.
\end{proof}

\subsection{Induced Clifford $C^*$-algebra}\label{ind.Cli.alg.}
Assume that 
$F = l+c: H' \to H$ is 
 finitely approximable as in Definition \ref{fin-appro}
 with respect to the data
$ W'_0   
 \subset   \dots  \subset  W'_i   \subset  \dots  \subset H' $
 with open disks
 $D_{r_i}' \subset W_i'$ and  $D_{s_i} \subset W_i$, and 
$F_i =l_i +c_i: W_i' \to W_i$.

 Let $S_r  =C_0(-r,r) \subset S$ be the set of   continuous functions on $(-r,r)$
 vanishing at infinity, and 
 consider the following $C^*$-subalgebras
$$  S_{r_i} \hat{\otimes} C_0(D'_{r_i}, Cl(W'_i)) 
\equiv S_{r_i} {\frak C}(D'_{r_i}).$$
Since the inclusion
$F_i^{-1}(D_{s_i}) \subset D_{r_i}'$ holds, 
it induces a $*$-homomorphism
$$F_i^*: S_{s_i} {\frak C}(D_{s_i})
\to S_{r_i} {\frak C}(D'_{r_i})$$
given by
$ F_i^*(h)(v') := \bar{l}_i^{-1}(h(F_i(v')))$.
Denote its image by
$$S_{r_i} {\frak C}_{F_i}(D'_{r_i}) :=
F_i^*(S_{s_i} {\frak C}(D_{s_i}))$$
which is a $C^*$-subalgebra with the norm
$||\quad||_{S_{r_i} {\frak C}_{F_i}}$.

Let us  
consider a family of elements
$$\alpha_i \in S_{r_i}{\frak C}_{F_i}(D_{r_i}' ), \qquad  i \geq i_0$$
 for some $i_0$.
Let us say that the family is $F$-{\em compatible} if
there is an element $u_{i_0} \in S_{s_{i_0}}{\frak C}(D_{s_{i_0}} )$
such that
$$\alpha_i = F_i^*(u_i) \in S_{r_i}{\frak C}_{F_i}(D_{r_i}' )$$
holds for any $i \geq i_0$,
where $u_i = \beta(u_{i_0} ) \in S_{s_{i}}{\frak C}(D_{s_{i}} )
$ with the standard Bott map $\beta$.

\begin{rem}\label{beta_F-formula}
Consider the induced Clifford operator
$$C_{F_i} \equiv \bar{l}_i^{-1} \circ F_i: D_{r_i}' \to  W_i \hookrightarrow Cl(W_i')$$ 
  and
introduce a  $*$-homomorphism
$$\beta_{F_i}: S_{r_i} \to S_{r_i} {\frak C}(D'_{r_i})$$
 defined by
$\beta_{F_i}: f \to f(X \hat{\otimes} 1+ 1 \hat{\otimes} C_{F_i})$
by functional calculus.

Then, 
$F_i^* (\beta(f ) ) = \beta_{F_i}(f)
$,  for all $f \in S_{r_i}$
\end{rem}

For an element 
$ \alpha_i \in S_{r_i}{\frak C}_{F_i}(D_{r_i}' )$ and for  $i_0 \leq i$,
we denote its restriction
$$\alpha_i| D_{r_{i_0}}' \in S_{r_i} \hat{\otimes} C_b(D_{r_{i_0}}')
\hat{\otimes} Cl(W'_i).$$
Note that the norms satisfy the inequality
$$||\alpha_i||_{S_{r_i}{\frak C}_{F_i}(D_{r_i}' )}
 \  \geq \  ||\alpha_i|D_{r_{i_0}}' ||$$
where the right-hand side is the restriction norm.

\vspace{2mm}

 For an $F$-compatible  sequence 
 $\alpha = \{\alpha_i\}_{i \geq i_0}$, 
  the limit
$$|| \ \{\alpha_i\}_i \ || :=  \lim_{j \to \infty}
\lim_{i  \to \infty} \
 ||\alpha_i |D_{r_{j}}'||$$
exists because
  both $F_i$ and $l_i$ converge weakly (see Definition \ref{fin-appro}).
 Moreover
both $F_i^*$ and $\beta$ are
$*$-homomorphisms between $C^*$-algebras and so
are both norm-decreasing.

\begin{defn} \label{induced Clifford}
Let $F$ be finitely approximable.
The induced Clifford $C^*$-algebra is given by
$$S{\frak C}_F(H') = \overline{  \{ \ \{ \alpha_i\}_i ;  \ 
F\text{-compatible sequences } \} }
$$
which is obtained by the norm closure of all $F$-compatible sequences,
where the norm is  the above one.
\end{defn}

\begin{lem}
$(1)$  In the model case, 
$S{\frak C}_F(H')$ 
 coincides with $S{\frak C}_F(H')$
  in subsection \ref{Model} 
 
 $(2)$  When $F=l$ is asymptotically unitary,
 there is a natural $*$-isomorphism
$$\Phi: S{\frak C}_F(H') \cong S{\frak C}_l(H')$$
where the right-hand side is given  in Definition \ref{SC_l}.

$(3)$
Suppose $F$ is $\Gamma$-finitely approximable. Then 
 $S{\frak C}_F(H')$ is independent of choice of 
$\Gamma$-finite approximations.
\end{lem}

\begin{proof}
One can choose $W_i' = E_1 \oplus \dots \oplus E_i'$.
Then $(1)$ 
 follows from the equality
$$F_{i}^* \circ \beta  = \beta_{F_{i}}:
S_{s_i} \to S_{r_i}{\frak C}_{F_i}(D'_{r_i})$$
by Remark \ref{beta_F-formula},
with Lemma \ref{cpt-supp}.

Let us consider $(2)$, and set $F =l$.
Recall the  Bott map which is given above of the Definition   \ref{SC_l},
and denote it as
$$\beta_l : S {\frak C}_l(W_i') \to 
S {\frak C}_l(H').$$

Let $\{l_i\}_i$ be asymptotically unitary, and denote
$l_i (W_i)= \tilde{W}_i$ and $l(W_i) =W_i$.
For each $i_0$ and $\epsilon >0$, there is some $i_0' >> i_0$ such that 
$$||\text{pr}_{i_0,i} - \text{ id } || < \epsilon$$
holds for any $i \geq i_0'$,
 where $\text{pr}_{i_0,i}: W_{i_0} \to l_i(W_{i_0}) $
is the orthogonal projection.
Let $\text{\bf pr}_{i_0,i}: W_{i_0} \cong l_i(W_{i_0})$
be the unitary of the polar decomposition.

Take an element $\{\alpha_i \}_i \in  S{\frak C}_F(H')$, with 
$\alpha_i = l_i^*(u_i)$ and  
$u_i = \beta(u_{i_0}) \in S{\frak C}(\tilde{W}_i)$.
Note that the restriction
$\beta(u_{i_0})|_{W_{i_0}} = u_{i_0}$ holds.
Then  by the condition of asymptotic unitarity, 
the restriction of their difference satisfies the estimate
$$|| l^*(\beta(\text{\bf pr}_{i_0,i}^*(u_{i_0}))) 
- \alpha_i||_{W'_{i_0}} < \epsilon
$$
where $\beta(\text{\bf pr}_{i_0,i}^*(u_{i_0})) \in 
S{\frak C}(W_{i})$.
Then we set
$$\Phi(\{\alpha_i \}_i) = \lim_{i \to \infty} \ 
 \beta_l(l^*(\beta(\text{\bf pr}_{i_0,i}^*(u_{i_0}))) )
.$$
$\Phi$ is norm-preserving,  so it extends to an injective
 $*$-homomorphism from $ S{\frak C}_F(H')$.

Let us verify that it is surjective.
One can follow in a converse way to the above.
Take an element 
$\delta = \beta_l(\delta_{i_0}) \in S{\frak C}_l(H')$ with 
$\delta_{i_0} \in S{\frak C}_l(W'_{i_0})$,
and set $\delta_i = l^*(\beta(\delta_{i_0})) $.
Let us set
$w_{i_0} = (l^*)^{-1}(\delta_{i_0}) \in S{\frak C}(W_{i_0})$ by 
$v \to \bar{l}(\delta_{i_0}(l^{-1}(v)))$. Then we set
$$u_{i_0} = (\text{\bf pr}_{i_0,i}^{-1})^*(w_{i_0} )\in 
S{\frak C}(\tilde{W}_{i_0})$$
 and
$\alpha_i = l_i^*(\beta(u_{i_0})) \in S{\frak C}_{l_i}(W_i')$.
The restriction of their difference satisfies the estimate
$$|| l^*(\beta(w_{i_0}))
- \alpha_i||_{W'_{i_0}} < \epsilon
$$
where $\beta(w_{i_0}) \in S{\frak C}(W_{i})$.
The estimate
$|| l^*(\beta(w_{i_0})) - \delta_i ||_{W'_{i_0}} < \epsilon$ 
is satisfied because
 $|| l^*(w_{i_0}) - \delta_{i_0}|| < \epsilon$ holds.
This implies that $\Phi(\{\alpha_i\}_i) = \delta \in S{\frak C}_l(H')$.

Hence $\Phi$  is an isometric $*$-homomorphism with dense image.
This implies that it is surjective.

Let us verify the last property $(3)$.
 Choose any subindices $j_i \geq i$ for $i =1,2, \dots$,
 and 
 consider the sub-approximation 
 given by the data $\{F_{j_i} \}_i$.
 If we replace the original data $\{F_i\}_i$
 by this subdata,  still we obtain the same $C^*$-algebra
 $S{\frak C}_F(H')$ as their norms coincide as follows:
 $$\lim_{i  \to \infty} ||\alpha_i |D_{r_{j_{i_0}}}'||
  =
 \lim_{i  \to \infty} ||\alpha_{j_i}|D_{r_{j_{i_0}}}'||.$$
Let us take two $\Gamma$-finite approximations
and denote them by
$F_i^l: (D_{s_i}')^l \to W_{i,l}$
for $l =1,2$.

Take an $F$-compatible  sequence 
 $\alpha = \{\alpha_i\}_{i \geq i_0}$ with respect to 
 $F_i^1: (D_{r_i}')^1 \to W_{i,1}$,
 where $\alpha_i = (F_i^1)^*(u_i)$
 and $u_i = \beta(u_{i_0}) \in S_{s_i}{\frak C}(D^1_{s_i})$. 
Let us take subindices  $j_i \geq i$ for $i =1,2, \dots$
so that $\lim_{i \to \infty} \ d'(W_{j_i,1} , W_{i,2})=0$ holds
(see \ref{auo}).

Let us set $\alpha_i' = (F_i^2)^*(u_i)$.
Then it follows from the definition of $F$-compatible  sequence that 
the convergence
$$\lim_{i  \to \infty} ||\alpha_{j_i} |D_{r_{j_{i_0}}}' - 
  \alpha'_{i}|D_{r_{j_{i_0}}}'|| =0$$
holds.
Combining  this result with the above,  we obtain the desired conclusion.
\end{proof}

\begin{lem}
If $F$ is $\Gamma$-finitely approximable, then 
there is a canonical $\Gamma$-action on 
$S{\frak C}_F(H')$.
\end{lem}

\begin{proof}
Recall that
if  $\{ W'_i, F_i, l_i\}_i$ is a finite approximation data,
then so is $\{ \gamma(W'_i), \gamma^*(F_i), \gamma^*(l_i)\}_i$
(see the last sentence in section \ref{fin-dim-appr}).

Take an $F$-compatible  sequence 
 $\alpha = \{\alpha_i\}_{i \geq i_0}$ with respect to 
 $F_i: D_{r_i}' \to W_{i}$,
 where $\alpha_i = (F_i)^*(u_{i})$
 and $u_i = \beta(u_{i_0}) \in S_{s_i}{\frak C}(D_{s_i})$. 
Then
$\{\gamma^* \alpha_i\}_i$ is an $\gamma^*(F)$-compatible
sequence as,  for $m' \in \gamma(D'_{r_i})$, 
\begin{align*}
\gamma^*(\alpha_i) (m') & = \gamma^*((F_i)^*(u_{i}))(m')
= \gamma u_i(F_i(\gamma^{-1}(m'))) \\
& = \gamma \beta(u_{i_0})(F_i(\gamma^{-1}(m')))
= \beta( \gamma u_{i_0})(F_i(\gamma^{-1}(m'))) \\
& = \beta( \gamma u_{i_0} \gamma^{-1})( \gamma F_i(\gamma^{-1}(m')))
= (\gamma^*F_i)^*(\gamma^*(u_i))(m').
\end{align*}
Thus,  $\gamma^*(\alpha_i) = (\gamma^*F_i)^*(\beta(\gamma^*(u_{i_0})))$
holds.
\end{proof}

\section{Higher degree $*$-homomorphism}

Let $F = l+c: H' \to H$ be a $\Gamma$-equivariant nonlinear map, 
whose linear part $l$
gives an isomorphism.
 For a finite-dimensional linear subspace $V \subset H $, 
denote the orthogonal projection by
$\text{pr}_V: H \to V$. For $V' =  l^{-1}(V)$,
 we have the modified map
$$F_V = l+ \text{pr}_V \circ c : V' \to V.$$
The restriction  map
$F_V \to F_{U}$ satisfies the formula
$F_{U}= \text{pr}_{U} \circ F_V |_{U}$,
for a   linear  subspace $U \subset V$.

Our initial  idea was to pull back 
$W_i =l(W'_i)$ by $F_{W_i}$ and combine them all.
For $F_i = \text{pr}_i \circ F  : W_i' \to W_i$, consider the induced $*$-homomorphism
$F_i^* : S{\frak C}(W_i)  \to   S {\frak C}(W_i')$.
Let us explain how  difficulty arises if one tries to obtain
a $*$-homomorphism  in this way.
For simplicity, assume  
$l$ is unitary and the image of $c$ is contained in a finite-dimensional linear subspace $V \subset H$.
This will be the simplest situation but already some difficulty appears when we try  to construct
the induced $*$-homomorphism by $F$.

Assume  $F$ is metrically proper.
This is equivalent to saying  that the restriction $F: V' \to V$ is proper in this particular situation, 
where $V' =l^{-1}(V)$ is the finite-dimensional linear subspace.
 Let us consider  the diagram
$$\begin{CD}
  S{\frak C}(W_i)  @> F^*_i >>  S {\frak C}(W_i')  \\
  @V  \beta  VV@V   \beta_l  VV \\
  S{\frak C}(W_{i+1})  
@> F_{i+1}^*  >> S{\frak C}(W_{i+1}') 
\end{CD}$$
This diagram is far from commutative as
 the following map
$$c : (W_i')^{\perp} \cap W'_{i+1} \to V$$
can affect to control the behavior of $F$ as $i \to \infty$.
Thus,  the induced maps by $F_i^*$ will not converge
in $S{\frak C}(H')$
 in general.
This is a point where we have account for the  nonlinearity  
of $F$
to construct the target  $C^*$-algebra, and is the reason  we have to 
use $S{\frak C}_F(H')$ instead of $S{\frak C}(H')$ below.

\subsection{Degree of proper maps}
Let $E',E$ be  two  finite-dimensional vector spaces,
and $F =l+c: E' \to E$ be a proper smooth map whose
 linear part $l: E' \cong E$ gives
an isomorphism.

Let us reconstruct the degree of $F \in {\mathbb Z}$ by use of $l$.
Let $\bar{l}: E' \to E$ be the unitary corresponding to the polar decomposition.
 Then, $\bar{l}$  induces the  algebra isomorphism
$\bar{l}: Cl(E') \cong Cl(E)$, and 
 we have the induced $*$-homomorphism
\begin{align*}
& F^*: S{\frak C}(E)  = C_0(E, Cl(E)) \to S{\frak C}(E')   \\
& F^*(h)(v) = \bar{l}^{-1}(h(F(v))).
\end{align*}
Recall $S{\frak C}_F(E') = F^*(S{\frak C}(E) )$.
Then $F^*$ can be described as a  $*$-homomorphism
$$F^* : S{\frak C}(E) \to S{\frak C}_F(E') .$$

Let us consider the induced  homomorphisms 
between $K$-groups
$$\begin{CD}
  K_1(S{\frak C}(E))  @> F^* >>   K_1(S{\frak C}_F(E'))  @> \text{inc}_*
>> K_1(S{\frak C}(E'))  \\
@A  \beta AA   @. @AA  \beta A \\
 K_1(C_0({\mathbb R}))  @. @. K_1(C_0({\mathbb R}))
\end{CD}$$
where both $\beta$ give the   isomorphisms by \ref{HKT}.

Let 
$\tilde{F}^*: K_1(C_0({\mathbb R})) \to K_1(C_0({\mathbb R}))$ 
be the homomorphism determined uniquely  so that the diagram commutes. 
Let us equip  orientations on both $E'$ and $E$
so that $l$ preserves them.

\begin{lem} Passing through the isomorphism
$K_1(C_0({\mathbb R})) \cong   {\mathbb Z}$, 
$$\tilde{F}^* : {\mathbb Z} \to {\mathbb Z}$$
is given by multiplication by the degree of  $F$.
\end{lem}
\begin{proof} 
{\bf Step 1:}
Let us consider the composition of $*$-homomorphisms
$$C_0(E, Cl(E)) \to C_0(E', Cl(E')) \cong C_0(E, Cl(E))$$
where the first map is $F^*$ and the second map  is given by
$$(\bar{l}^{-1})^* (h') (v ) \equiv \bar{l}(h'(\bar{l}^{-1}(v))).$$
The latter gives an isomorphism since $l$ is  isomorphic.
Thus,  it is sufficient to see the conclusion for the composition.
The composition is given by
$$h \to \{ \ v \mapsto  h(F \circ \bar{l}^{-1}  (v))\}$$

{\bf Step 2:}
Let $l_t: E' \cong E$ be another family of linear  isomorphisms
with  $l_0=\bar{l}$ and $l_1=l$.
It  induces a family of $*$-homomorphisms
\begin{align*}
&  F^*_t : C_0(E, Cl(E)) \to  C_0(E, Cl(E)) \\
& \qquad \qquad h \to \{ \ v \to  h(F \circ (l_t)^{-1}  (v))\}.
\end{align*}
Since homotopic $*$-homomorphisms induce the same maps between their $K$-groups, 
it is sufficient  to see the conclusion for $F^*_1$.
Noting the equality $F \circ l^{-1}  = 1 + c \circ  l^{-1} $, 
it is enough to assume $l $ is the identity.

{\bf Step 3:}
When $l$ is the identity,    
 $F^*: S{\frak C}(E) \to S{\frak C}(E) $ 
 is given by 
 $$\text{ id } \times F^* : S{\frak C}(E)  \cong (S \hat{\otimes} Cl(E)) \otimes C_0(E) 
 \to  (S \hat{\otimes} Cl(E)) \otimes C_0(E) $$
whose induced homomorphim on a $K$-group is given by degree $F$,
passing through the isomorphism
$$K_1(S \hat{\otimes} Cl(E)\otimes C_0(E)) \cong K_1(S) \cong
K_*(C_0(E))  \cong {\mathbb Z}$$
where $*$ is $0$ or $1$ with respect to whether $\dim E$ is even or odd.
The first isomorphism comes from Proposition \ref{HKT},
 and the second  is 
the classical Bott periodicity (see \cite{atiyah}).
\end{proof}


\subsection{Induced map for a strongly finitely approximable map}
Let $F = l+c : H' \to H$ be a strongly finitely approximable map.
There are finite-dimensional linear subspaces
$W_i' \subset W_{i+1}' \subset \dots \subset H'$ whose union is dense, 
such that the compositions with the projections 
$\text{pr}_i \circ F: W_i' \to W_i =l(W_i)$ consist of a finitely approximable 
data with the constants $r_i, s_i \to \infty$.

Let us consider the restriction
$$F_{i+1} : D'_{r_i} \cap W_{i+1}' \to W_{i+1}.$$
Decompose
$W_i' \oplus U_i' =  W_{i+1}' $, and  define $F_{i+1}^0: W_{i+1}' \to W_{i+1}$  
by
$$F_{i+1}^0 (w'+u')=  F_{i}(w') + l(u').$$
Then by definition,  the estimate 
$$\sup_{m \in D'_{r_i} \cap W_{i+1}'} \ ||F_{i+1} (m) - F_{i+1}^0(m)|| < \delta_i$$
holds, where $0 < \delta_i \to0$.

\begin{sublem} \label{bott-commu}
Suppose $l: H' \cong H$ is unitary. 
Let $\beta: S{\frak C}(W_{i}')  \to S{\frak C}(W_{i+1}') $ be the Bott map.
Then the equality holds
$$\beta \circ F_i^* =  (F_{i+1}^0)^* \circ \beta: S{\frak C}(W_{i})  \to S{\frak C}(W_{i+1}') .$$
\end{sublem}

\begin{proof}
Take $f \hat{\otimes} h \in S{\frak C}(W_{i})$ with
$(\beta \circ F_i^*)(f \hat{\otimes} h) = \beta( f) \hat{\otimes}F_i^*(h)$.
By contrast, 
$\beta(f \hat{\otimes} h) = \beta(f ) \hat{\otimes} h $ and, hence,
\begin{align*}
 (F_{i+1}^0)^* \circ \beta( f \hat{\otimes} h) & =
  (l \oplus F_i )^* \circ  \beta( f) \hat{\otimes} h 
   = l^*(\beta(f)) \hat{\otimes} F_i^*(h)
 \end{align*}
 where $l^* : S{\frak C}(U_i) \cong S{\frak C}(U'_i)$ with $U_i = l(U_i')$.
Since $l$ is unitary, the equality holds
$$ l^*(\beta(f)) = \beta(f) \in S{\frak C}(U'_i).$$
\end{proof}

\begin{prop}\label{cpt-degree}
Let $F = l+c : H' \to H$ be a strongly finitely approximable map.
Then  the family $\{ F_i^*\}_i$  induces a 
 $*$-homomorphism
$$F^* : S{\frak C}(H)\to  S{\frak C}(H').$$
\end{prop}

\begin{proof} 
{\bf Step 1:}
Let us take an element $\alpha  \in S{\frak C}(H)$ and 
its approximation 
$\alpha_i \in S_{r_i} {\frak C}(D_{r_i})$ with 
$\lim_{i \to \infty} \ \beta(\alpha_i) = \alpha \in S{\frak C}(H)$ 
by Lemma \ref{cpt-supp}.

Assume  $l :H' \cong H$ is unitary, and consider the following  two elements:
$$\beta(F^*_i(\alpha_i) ), \ F^*_{i+1}(\alpha_{i+1}) 
  \in S{\frak C}(W_{i+1}').$$
Then by Sublemma \ref{bott-commu} we have the estimates
\begin{align*}
||\beta(F^*_i & (\alpha_i) )  -  F^*_{i+1}(\alpha_{i+1}) ||  =
||(F^0_{i+1})^* ( \beta(\alpha_i) ) -  F^*_{i+1}(\alpha_{i+1}) || \\
& ||(F^0_{i+1})^* ( \beta(\alpha_i) )- F_{i+1}^* ( \beta(\alpha_i) ) +
||  F_{i+1}^* ( \beta(\alpha_i) )-  F^*_{i+1}(\alpha_{i+1}) || \\
& \leq \delta_i || \beta(\alpha_i) || + || \beta(\alpha_i) - \alpha_{i+1}||
\end{align*}
The first term on the right-hand side converges to zero since 
$ || \beta(\alpha_i) ||$
are uniformly bounded with $\delta_i \to 0$.
The second term also converges to zero.
Thus,  the $*$-homomorphisms asymptotically commute
 with the Bott map. Hence,
 the sequence $\beta(F^*_i  (\alpha_i) )  \in  S{\frak C}(H')$ converges, and 
 gives a $*$-homomorphism
 $F^* : \alpha \to F^*(\alpha) := \lim_i \ \beta(F^*_i  (\alpha_i) ) $.
 Clearly this assignment is independent of the  choice of 
 approximations of $\alpha$.

{\bf Step 2:}
Let us  consider the case when $l$ is not necessarily unitary, 
but is asymptotically unitary.

Let $\beta_l : S \to S{\frak C}_l(U_i')$ be the variant of the Bott map 
in \ref{var-bott}.  Then the same argument to Sublemma \ref{bott-commu} verifies the equality
$$\beta_l \circ F_i^* =  (F_{i+1}^0)^* \circ \beta: S{\frak C}(W_{i})  \to S{\frak C}_l(W_{i+1}') .$$
Hence the parallel estimate to step $1$ above verifies
that   
 the sequence converges
$$\beta_l(F_i^*(\alpha_i)) \in  S{\frak C}_l(H').$$
This also
gives a $*$-homomorphism
 $F^* : \alpha \to F^*(\alpha) := \lim_i \ \beta_l(F^*_i  (\alpha_i) ) $.
As $S{\frak C}_l(H') \cong S{\frak C}(H')$ are $*$-isomorphic 
by Proposition \ref{iso}, we obtain the desired $*$-homomorphism.
\end{proof}

\begin{rem}
Suppose  $F=l+c$ satisfies the conditions 
to be strongly finitely approximable, except that 
$l$ is not necessarily isomorphic, but the Fredholm index is zero.

We can  still construct the induced $*$-homomorphism
$F^* : S{\frak C}(H)\to  S{\frak C}(H')$ as below.

There are finite-dimensional linear subspaces
$V' \subset H'$ and $V \subset H$ such that
the restriction gives an isomorphism
$l : (V')^{\perp} \cong V^{\perp}$,
where $V^{\perp} \subset H$ is the orthogonal complement.
Choose any unitary
$l' : V' \cong V$ and take their sum
$$L \equiv l \oplus l' : (V')^{\perp} \oplus V'  \cong V^{\perp} \oplus V.$$

Let us use $L$ to pull back the Clifford algebras
and use $F$ itself to pull back the functions.
Then we can follow from step $1$ and step $2$ in the same way.
\end{rem}

\begin{defn}
Let $F: H' \to H$ be a strongly finitely approximable  map. 
Then the induced map
$$F^*: K_1(S{\frak C}(H)) \cong {\mathbb Z} \to K_1(S{\frak C}(H')) \cong {\mathbb Z}$$
is given by   multiplication by an integer degree $F \in \mathbb Z$.
We call it the $K$-theoretic  degree of $F$.
\end{defn}

\subsection{Induced map for $\Gamma$-finitely approximable map}
Let us start from a general property, and let
$H$ be a Hilbert space with exhaustion 
 $W_0 \subset \dots \subset W_i \subset  \dots \subset H$
 by finite-dimensional linear subspaces.
  Choose
divergent numbers $r_i < r_{i+1}< \dots \to \infty$, and
denote $r_i$ balls  by $D_{r_i} \subset W_i$.
 Let $S_r  =C_c(-r,r) \subset S$ be the set of  compactly supported continuous functions on $(-r,r)$.

The following restates Lemma \ref{cpt-supp}
\begin{lem}\label{cpt-supp-2}
 For any $ \alpha \in S{\frak C}(H)$, there is 
  a family
   $$\alpha_i \in S_{r_i} \hat{\otimes} C_0(D_{r_i}, Cl(W_i)) := S_{r_i} {\frak C}(D_{r_i})$$
such that their images by the Bott map converge to $\alpha$
$$\lim_{i \to \infty} \ \beta(\alpha_i) = \alpha \in S{\frak C}(H).$$
\end{lem}

\subsubsection{Induced $*$-homomorphism} 
Let $H', H$ be Hilbert spaces on which $\Gamma$ act linearly and isometrically, and
let  $F = l+c: H' \to H$ be a $\Gamma$-equivariant map
such that $l: H' \cong H$ is a linear isomorphism.

Assume that 
$F$ is $\Gamma$-finitely approximable so that   there is
a  family of finite-dimensional linear subspaces
$$ W'_0   \subset   W'_1 
 \subset   \dots  \subset  W'_i   \subset  \dots  \subset H' $$ 
 with dense union,
 and a  family of maps   $F_i  : W'_i \to W_i =l_i(W_i')$
with the inclusions $F_i^{-1}(D_{s_i}) \subset D_{r_i}'$.
Moreover the following 
convergences hold for each $i_0$:
\begin{align*}
& \lim_{i \to \infty} \  \sup_{m \in D_{r_{i_0}}'} \ ||F(m) - F_i(m)|| =0 \qquad  (*_1) \\
& \lim_{i \to \infty}  \  
 \ || \ (l  -l_i)|W_{i_0}'   \ || =0 \qquad (*_2).
\end{align*}
    Recall the induced $*$-homomorphism
  $$F_i^*: S{\frak C}(D_{s_i}) \to S{\frak C}_{F_i}(D_{r_i}') $$ 
   and the induced Clifford $C^*$-algebra 
   $S{\frak C}_{F}(H') $ in Definition \ref{induced Clifford}.

  \vspace{3mm}

\begin{thm}\label{cov-degree}
Let $F =l+c: H' \to H$ be  $\Gamma$-finitely approximable.

Then it induces  the   equivariant $*$-homomorphism
$$F^* :  S{\frak C}(H)   \to  S{\frak C}_F(H').$$
\end{thm}

\begin{proof}
Let us take an element $v \in S{\frak C}(H)$ and
its  approximation $v = \lim_{i \to \infty} v_i$
with $v_i \in   S_{s_i}{\frak C}(D_{s_i})
= C_0(-s_i,s_i) \hat{\otimes}C_0(D_{s_i} , Cl(W_i))$.

Let us  recall
the $*$-homomorphism  in \ref{ind.Cli.alg.}
 $$  F_i^*: 
 S_{s_i}{\frak C}(D_{s_i}) \to   S_{r_i}{\frak C}_{F_i}(D'_{r_i}) .$$
Let us fix $i_0$ and let
$u_i = \beta(v_{i_0}) \in S_{s_i}{\frak C}(D_{s_i}) $
be the image of the standard Bott map.
Then the family
$$\{F_i^*(u_i)\}_{i \geq i_0}$$
determines an element in $ S{\frak C}_{F}(H')$, which gives 
 a $*$-homomorphism
$$F^*: S_{s_{i_0}}{\frak C}(D_{s_{i_0}}) \to  S{\frak C}_{F}(H')$$
since both $F_i^*$ and $\beta$ are $*$-homomorphisms.
Note that the composition of  two $*$-homomorphisms
$$\begin{CD}
S_{s_{i_0}}{\frak C}(D_{s_{i_0}}) @> \beta >>
S_{s_{i_0'}}{\frak C}(D_{s_{i_0'}}) @> F^* >>
S{\frak C}_{F}(H')
\end{CD}$$
coincides with $F^*: S_{s_{i_0}}{\frak C}(D_{s_{i_0}}) 
\to S{\frak C}_{F}(H')$.

For a small $\epsilon >0$, 
take  two sufficiently large $i_0'  \geq i_0 >>1$
such that the estimate
$||\beta(v_{i_0}) - v_{i_0'}|| < \epsilon$ 
holds,
and set  $u_i' = \beta(v_{i_0'}) \in
S_{s_i}{\frak C}(D_{s_i})$ for $i \geq i_0'$.
Since $F^*$ is 
 norm-decreasing,
the estimate
$
 ||F_i^*(u_i) -F_i^*(u'_{i})|| < \epsilon$
 holds for all $i \geq i_0'$.
Hence, the estimate
$$ ||F^*(v_{i_0}) -F^*(v_{i_0'})|| < \epsilon$$
holds.

Thus, we obtain the assignment
 $v \to \lim_{i_0 \to \infty} F^*(v_{i_0})$,
which
 gives 
 a $\Gamma$-equivariant $*$-homomorphism
$$F^*: S{\frak C}(H)   \to S{\frak C}_F(H')$$
where $\{v_i\}_i$ is any approximation of $v$.
\end{proof}

\begin{defn} 
Let $F: H' \to H$ be a $\Gamma$-finitely approximable map. 
 Then, the higher degree  of $F$
 is given by the induced  homomorphism
  $$F^*: K_{*+1}(C^*(\Gamma)) \to 
  K_{*}(S{\frak C}_F (H')\rtimes \Gamma).$$
   \end{defn}

     \vspace{3mm}
     
     \section{Computation of $K$-group of induced Clifford $C^*$-algebras}
We compute the equivariant $K$-group of  induced Clifford $C^*$-algebras 
for some particular cases. This   can be a simple model case for further computation of the groups.

\subsection{Basics}
Let us collect some of basics  which we will need.
We start from some analytic aspects of Sobolev spaces.
We denote by $W^{k,2}$ as the Sobolev $k$-norm
which is a linear subspace of  $L^2$.
It is a Hilbert space and, hence, complete 
by the norm which involves  derivatives up to  the
$k$-th order, and 
  incomplete with respect to the $L^2$ inner product
for $k \geq 1$.


The following is well known.

\begin{lem} \label{Sob}
Suppose $k \geq 1$. Then

$(1)$ 
  The multiplication
 $$W^{k,2}(S^1) \otimes W^{k,2}(S^1) \to W^{k,2}(S^1)$$
 is compact on each bounded set.
 
 $(2)$   The continuous embedding
  $W^{k,2}(S^1) \hookrightarrow C^0(S^1)$
  holds.
  \end{lem} 
     
 In particular an element in $W^{k,2}(S^1)$ can be regarded
 as a continuous function.

 Later we will consider the non linear map
 $$F: W^{k,2}(S^1) \to W^{k,2}(S^1)$$
 by $F(a) = a+a^3$.

\begin{rem}
$(1)$
Let $A$ be a $C^*$-algebra on which a finite cyclic group
${\mathbb Z}_l$ acts. Then the crossed product is defined as
$A \rtimes {\mathbb Z}_l = \{ (a_g)_{g \in {\mathbb Z}_l} \}$
with their product by
$(a_g)(b_g) = (\sum_{g_1g_2=g \in {\mathbb Z}_l} a_{g_1}g_1(b_{g_2}))$.
It 
 induces the action by ${\mathbb Z}$ on $A$
 by using  the natural projection
 $\pi_l: {\mathbb Z} \to {\mathbb Z}_l$.
In such situation, there exists a six term exact sequence
between $K_*(A \rtimes {\mathbb Z})$ and 
$K_*(A \rtimes {\mathbb Z}_l)$. 
However  this does not seem to contain enough information to 
apply to our situation.
We proceed  in a direct way.
Recall that an element 
$a \in A \rtimes {\mathbb Z}$ can be approximated 
by $a' \in C_c({\mathbb Z}, A)$.

$(2)$
Let us  take an element $u \in K(A \rtimes {\mathbb Z})$ and
represent it by $u= [p] - [\pi(p)] $, where 
$\pi: \bar{A} = A \oplus {\mathbb C} \to {\mathbb C}$
 is the projection.
Recall that $[p]- [\pi(p)] =[q]-[\pi(q)] $, if and only if
there is some $v \in M_{n,m}(A \rtimes {\mathbb Z})$
such that
$$p \oplus 1_a = v^*v , \quad 
vv^* = q \oplus 1_b$$
for some $a,b \geq 0$.
\end{rem}

\subsection{Computation of equivariant $K$-group for a toy model}

\subsubsection{Finite cyclic  and finite-dimensional case}\label{fin-case}
Consider a $\mathbb{Z}_2$-equivariant map
$F: \mathbb{R}^2 \to \mathbb{R}^2$
by
$$\begin{pmatrix}
a \\
b
\end{pmatrix}
\to
\begin{pmatrix}
a +b^3 \\
b+a^3
\end{pmatrix}$$
where the involution acts by the coordinate change.

We claim that this is proper of non-zero degree.
In fact,  if $a+b^3=0$, then  the equality
$b+a^3 = b -b^{3^2}$ implies properness.

Consider a $\mathbb{Z}_2$-equivariant perturbation
$$F_t
\begin{pmatrix}
a \\
b
\end{pmatrix}
= \begin{pmatrix}
t a +b^3 \\
tb+a^3
\end{pmatrix}$$
for $t \in (0, 1]$.
If $ta+b^3=0$, then 
$tb + a^3= tb - t^{-3}b^{3^2}$.
Thus,  this is a family of proper maps.
At $t=0$, 
$F_0 : \mathbb{R}^2 \to  \mathbb{R}^2$ is a proper map of degree $-1$,
since it is again $\mathbb{Z}_2$-equivariantly proper-homotopic to the involution
$$I: 
\begin{pmatrix}
a \\
b
\end{pmatrix} \to
\begin{pmatrix}
b \\
a
\end{pmatrix}.$$
Note that  it becomes degree zero,
if we replace the exponent  $3$ by $2$.

Next, we  generalize slightly as follows.
Consider a $\mathbb{Z}_l$ equivariant map
$F: \mathbb{R}^l \to \mathbb{R}^l$
by
$$\begin{pmatrix}
a_1 \\
\dots \\
a_l
\end{pmatrix}
\to
\begin{pmatrix}
a_1 +a_l^3 \\
a_2+a_1^3 \\
\dots \\
a_l +a_{l-1}^3
\end{pmatrix}$$
where the action is given by cyclic permutation of the coordinates.
By the parallel argument as above, 
this turns out to be  a proper map.
To compute its degree, 
consider a perturbation
$$\begin{pmatrix}
a_1 \\
\dots \\
a_l
\end{pmatrix}
\to
\begin{pmatrix}
ta_1 +a_l^3 \\
ta_2+a_1^3 \\
\dots \\
ta_l +a_{l-1}^3
\end{pmatrix}$$
for $t \in (0, 1]$.
This is a family of $\mathbb{Z}_l$-equivariant proper maps, and 
at $t=0$, 
$F_0 : \mathbb{R}^l \to  \mathbb{R}^l$ is a  proper map of degree $\pm1$,
determined by the parity of $l$.
In fact there is a ${\mathbb Z}_l$-equivariant 
proper-homotopy $F^l_t$ to the cyclic permutation
$$T_l \equiv F^l_0 
\begin{pmatrix}
a_1 \\
a_2 \\
\dots \\
a_l
\end{pmatrix}
\to
\begin{pmatrix}
 a_l \\
a_1 \\
\dots \\
a_{l-1}
\end{pmatrix}.$$

\begin{cor}
$F$ induces  a ${\mathbb Z}_l$-equivariant   isomorphism
$$K_1^{{\mathbb Z}_{l}}(S{\frak C}_{F}({\mathbb R}^l)) 
\cong K_1^{{\mathbb Z}_{l}}(S{\frak C}({\mathbb R}^l)) 
\cong  R({\mathbb Z}_{l})$$
on the equivariant $K$-theory.
 \end{cor}

 \begin{proof}
  $F$ is $\mathbb Z_l$-equivariantly properly
 homotopic to $T_l$ above, 
 and so the isomorphism
 $K_1^{{\mathbb Z}_{l}}(S{\frak C}_{F}({\mathbb R}^l)) 
 \cong K_1^{{\mathbb Z}_{l}}(S{\frak C}_{T_l}({\mathbb R}^l)) $ holds.

 $T_l$  is a $\mathbb Z_l$-equivariantly linear isomorphism because
  $\mathbb Z_l$ is commutative.
 It follows from Definition \ref{as-unitary}
 that a linear isomorphism between
  finite-dimensional vector spaces is asymptotically unitary. 
 Then by Proposition \ref{iso} that 
 $S{\frak C}_{T_l}(H')$ is $\mathbb Z_l$-equivariantly $*$-isomorphic to 
 $S{\frak C}(H')$. In particular we have the isomorphism
 $ K_1^{{\mathbb Z}_{l}}(S{\frak C}_{T_l}({\mathbb R}^l)) \cong
   K_1^{{\mathbb Z}_{l}}(S{\frak C}({\mathbb R}^l)) $.

The last isomorphism comes from HKT-Bott periodicity
for Euclidean space.
\end{proof}

\subsubsection{Infinite cyclic case}\label{inf.cyc.}
Let $H' =H$ be the closure of ${\mathbb R}^{\infty}$
with the standard metric.
It admits  an isometric action of ${\mathbb Z}$
 by the shift $T: H' \cong H'$
 $$T: (\dots, a_{-1}, a_0, a_1, \dots) \cong (\dots, a_{-2},a_{-1},a_0, \dots).$$
Then we consider the map $F: H' \to H $ by
$$F: 
\begin{pmatrix}
\dots \\
a_{-1} \\
a_0 \\
a_1 \\
\dots \\
\end{pmatrix}
\to
\begin{pmatrix}
\dots \\
a_{-1} +a_{-2}^3 \\
a_0+a_{-1}^3 \\
a_1+ a_0^3 \\
\dots \\
\end{pmatrix}.$$
If we restrict on ${\mathbb R}^{2l+1} \subset H'$
by $(a_{-l}, \dots,a_l) \to ( \dots, a_{-l}, \dots, a_l, 0, \dots)$, 
then its image is  in ${\mathbb R}^{2l+2} \subset H$. 
In fact
$$F: 
\begin{pmatrix}
\dots \\
0 \\
a_{-l} \\
a_{-l+1} \\
\dots \\
a_l \\
0 \\
\dots  \\
\dots
\end{pmatrix}
\to
\begin{pmatrix}
\dots \\
0 \\
a_{-l}  \\
a_{-l+1}+a_{-l}^3 \\
\dots \\
a_l+ a_{l-1}^3 \\
a_l^3 \\
0 \\
\dots
\end{pmatrix}.$$

Let us consider the map
$F^l : {\mathbb R}^{2l+1} \to {\mathbb R}^{2l+1}$ by
 $$F^l: 
\begin{pmatrix}
a_{-l} \\
a_{-l+1} \\
\dots \\
a_l \\
\end{pmatrix}
\to
\begin{pmatrix}
a_{-l}  + a_{l}^3\\
a_{-l+1}+a_{-l}^3 \\
\dots \\
a_l+ a_{l-1}^3 
\end{pmatrix}$$
which moves the last component to the first one.
In fact $F^l$ is still a proper map as presented in  \ref{fin-case}.

Let $W'_l = {\mathbb R}^{2l+1}$  be as above.
Then the data $(F^l, W'_l)$ 
gives the ${\mathbb Z}$-finite approximation in the sense of Definition \ref{fin-appro}.
 
 First, as in the finite cyclic case, we obtain 
 the   isomorphism
$$K_1^{{\mathbb Z}_{2l+1}}(S{\frak C}_{F^l}(W'_l)) 
\cong K_1^{{\mathbb Z}_{2l+1}}(S{\frak C}(W'_l))$$
on the equivariant $K$-theory.
 Notice that this isomorphism 
 heavily depends on the degree being equal to $\pm 1$.
 
 \begin{lem} The induced $*$-homomorphism
 $$(F^l)^*:  S{\frak C}(W_l) \to S{\frak C}_{F^l}(W'_l)$$
 is in fact, an isomorphism.
 \end{lem}
 
 \begin{proof} Injectivity follows from surjectivity of 
 $F^l$, because it has a non-zero degree.
 It has a closed range since it is isometric embedding.
Then, the conclusion follows since it has a dense range.
   \end{proof}
   
   Second, we obtain the inductive system
   $$\Phi_l \equiv 
   (F^{l+1})^* \circ \beta \circ ((F^l)^*)^{-1}: 
   S{\frak C}_{F^l}(W'_l) \to
   S{\frak C}_{F^{l+1}}(W'_{l+1}).$$
 By definition, the equality holds
 $$S{\frak C}_{F}(H') = \lim_l \ S{\frak C}_{F^l}(W'_l).$$
Forgetting the group action, we have the isomorphisms
 \begin{align*}
  K (S{\frak C}_F(H')) & =
 \lim_l \ K(S{\frak C}_{F^l}(W'_l))  
  \cong  \lim_l \ K(S{\frak C}(W'_l)) \\
& = K (S{\frak C}(H'))  \cong  K(S) \cong {\mathbb Z}
 \end{align*}
where   we  used the HKT-Bott periodicity.

 Now consider the group action by ${\mathbb Z}$.
 Let $F^l_t$ be the homotopy in subsection \ref{fin-case}, where
 $F^l_1=F^l$ and $F^l_0=T_l$.
 \begin{lem}\label{homot}
 There is a $*$ -isomorphism
 $$I_l: S{\frak C}_{F^l_0}(W'_l) \cong S{\frak C}_{F^l_1}(W'_l) .$$
 \end{lem}
 
 \begin{proof}
 In fact an element $u \in S{\frak C}_{F^l_0}(W'_l)$
 is expressed as $u = (F^l_0)^*(v_0)$ for some 
 $v_0 \in S{\frak C}(W_l)$. Because  $F^l_t$ has a non zero degree,
 it follows that $v_0$ is uniquely determined by $u$.
 Then assign $v_1 = (F^l_1)^*(v_0)$, and denote its map by
 $$I_l : S{\frak C}_{F^l_0}(W'_l) \cong S{\frak C}_{F^l_1}(W'_l).$$
 This is a $*$-homomorphism and,
 in fact, is an isomorphism, since  if we do the same thing,
 replacing the role of $F^l_0$ and $F^l_1$, then 
 we can recover $u$ again.
 \end{proof}
 
 \begin{prop}
 There is an isomorphism
 $$K_0(S{\frak C}_F(H') \rtimes {\mathbb Z})
\  \cong \ {\mathbb Z}.$$
 \end{prop}
 
 \begin{proof}
 Denote by $\bar{A}$  the unitization of $A$.
Take an element $u 
   \in K_0(A \rtimes {\mathbb Z})$  and 
represent it by $u= [p] - [\pi(p)] $.
Approximate 
$p \in Mat(\overline{S{\frak C}_F(H') \rtimes {\mathbb Z}})$
by  an element 
$p'  = (p'_g)_{g \in B} \in 
Mat(\overline{C_c({\mathbb Z}, S{\frak C}_F(H'))})$
where $B \subset {\mathbb Z}$ is a finite set.
There is some $l$ such that 
each $p_g'$ can be approximiated by $p_{g,l}' \in 
\overline{S{\frak C}_{F^l}(W'_l)}$.
Therefore, $p$ can be approximated by an element
$$p'' = (p_{g,l}') \in 
C(\{-l, \dots, l\}, Mat(\overline{S{\frak C}_{F^l}(W'_l)})).$$
Let us put
    $$ I_l^{-1}(p'') \in 
    C(\{-l, \dots, l\}, Mat(\overline{S{\frak C}_{T_l}(W'_l)}))$$
where $S_l=F^l_1$ is the cyclic permutation,
and $I_l$ is in Lemma \ref{homot}.

$$\tilde{p}'' \equiv \frac{I_l^{-1}(p'') +(I_l^{-1}(p''))^*}{2} 
 \in  Mat(\overline{S{\frak C}_T(H') \rtimes {\mathbb Z}})$$
is an  ``almost'' projection, in the sense that
$$||(\tilde{p}'')^2 - \tilde{p}''|| < \epsilon$$
for a small $\epsilon >0$.
Then there is a projection $\tilde{p} \in 
Mat(
\overline{S{\frak C}_T(H') \rtimes {\mathbb Z}})$
with the estimate
$$||\tilde{p} - \tilde{p}''|| < \epsilon'$$
for a small $\epsilon '>0$.

Now take another  representative  
 $u = [p]- [\pi(p)] =[q]-[\pi(q)] $.
 In the same way, we obtain a projection 
$\tilde{q} \in 
Mat(
\overline{S{\frak C}_T(H') \rtimes {\mathbb Z}})$.
 Recall that there is some $v \in 
 M_{n,m}(\overline{S{\frak C}_F (H') \rtimes {\mathbb Z}})$
such that
$$p \oplus 1_a = v^*v , \quad 
vv^* = q \oplus 1_b$$
for some $a,b \geq 0$.

Let $v' \in 
C(\{-l, \dots, l\}, Mat(\overline{S{\frak C}_{F^l}(W'_l)}))$
be another approximation and take $\tilde{v}'' \equiv I_l^{-1}(v')
\in  C(\{-l, \dots, l\}, Mat(\overline{S{\frak C}_{T_l}(W'_l)}))$.
Then we have the estimates
$$||(\tilde{v}'')^*\tilde{v}'' - \tilde{p} \oplus 1_a ||, \ \ 
||\tilde{v}''(\tilde{v}'')^* - \tilde{q} \oplus 1_b|| < \epsilon''$$
for a small $\epsilon'' >0$.
This implies the equality
 $$[\tilde{p}] -[\pi(\tilde{p})] = 
 [\tilde{q}] -[\pi(\tilde{q})] \in 
 K_0(S{\frak C}_T(H') \rtimes {\mathbb Z}).$$
 
 Therefore,  we obtain a well defined group homomorphism
 $$K_0(S{\frak C}_F(H') \rtimes {\mathbb Z})
 \to K_0(S{\frak C}_T(H') \rtimes {\mathbb Z}).$$
 If we replace the role of $F$ and $T$ and proceed
 in the same way as above, 
 we obtain another map 
 in a converse direction.
 By construction, their compositions are both the identities.
 Therefore,  this is an isomorphism on the $K$-groups.

 Since the translation shift 
 $T: H' \cong H'$ is unitary and $\mathbb Z$ is commutative,
 there is a $*$-isomorphism
 $$S{\frak C}_T(H') \rtimes {\mathbb Z} \cong
 S{\frak C}(H') \rtimes {\mathbb Z}.$$
 Passing through this isomorphism, we obtain the isomorphism
 $$K_0(S{\frak C}_F(H') \rtimes {\mathbb Z})
 \to K_0(S{\frak C}(H') \rtimes {\mathbb Z}).$$
 
 The right-hand side is isomorphic to
 $$K_1(C^* \mathbb Z) \cong K^1(S^1) \cong {\mathbb Z}$$ 
 by HKT.
 \end{proof}

\subsection{Nonlinear maps between
Sobolev spaces over the circle}
\subsubsection{Involution}\label{Sob-fin}
Consider the space 
$$S^1_2 = \mathbb{R} / 2 \mathbb{Z}= [0,2]/ \{0 \sim 2\}$$
and
$W^{k,2}(S^1_2)$ which  is generated by 
$ \sin( \pi k   s)$ and $ \cos( \pi k   s)$
for $k \in \mathbb{Z}$.

 Consider the Sobolev spaces
  $$W^{k,2}(S^1_2)_0 , \quad W^{k,2}(S^1_2)_1 \ \subset \ W^{k,2}(S^1_2)$$
   which  are generated by $W^{k,2}(0,1)_0$ and $W^{k,2}(1,2)_0$
  respectively.  Here,
  $W^{k,2}(S^1_2)_i $ is naturally isometric to 
  $W^{k,2}(S^1_2)_{i-1} $ by
  the shift operator
  $$T: u_1 \to u_0, \quad
  u_0(s) = u_1(s+1)$$ mod $2$.
  Note that $T^2$ is the identity.
    Therefore, we can identify both Hilbert spaces by the same symbol 
    $H$
  and, hence,  the following  inclusion holds:
   $$H \oplus H \subset W^{k,2}(S^1_2).$$

   We again consider the non linear map with $H'=H$
 $$F: H' \oplus H'
  \to H \oplus H$$
 by $F(a) = a+T(a)^3$, where the power is taken pointwisely.
Then, the map can be written as
$$\begin{pmatrix}
a \\
b
\end{pmatrix}
\to
\begin{pmatrix}
a +b^3 \\
b+a^3
\end{pmatrix}$$
As we have seen,  this is metrically proper.

Let $k=1$ for  simplicity of notation, and consider an element
$a \in W^{1,2}(S^1)$
$$a = \sum_{k= - \infty}^{\infty} \ 
a_k \sin(2 \pi k s)+b_k  \cos(2 \pi k  s)$$
and denote
$$a^3 = \sum_{k= - \infty}^{\infty} \ 
c_k \sin(2 \pi k s)+d_k  \cos(2 \pi k  s)$$

\begin{lem}\label{comp} Suppose 
 $||a||_{W^{1,2}} \leq r$. Then for any $\epsilon >0$, there is 
$n  = n(r,\epsilon)\geq 0$ such that the estimate holds
$$|| \sum_{|k| \geq n+1} \ 
c_k \sin(2 \pi k s)+d_k  \cos(2 \pi k  s)||_{W^{1,2}}
< \epsilon.$$
\end{lem}

\begin{proof}
It follows from  Lemma \ref{Sob} that 
$W^{1,2}(S^1) \to W^{1,2}(S^1)$ by  $a \to a^3$ 
is compact on each bounded set.
\end{proof}

Choose  divergent numbers as $\lim_i \ n_i  = \infty$.
For each $i \in \mathbb N$, 
let $V_i' \subset W^{1,2}(0,1)_0$ be the finite-dimensional linear subspace
spanned by $  \sin(2 \pi k s)$ and $ \cos(2 \pi k  s)$ for  $|k| \leq n_i$,
and set
$$W_i' = V_i' \oplus T(V_i') \subset H' \oplus H'.$$
Denote $\text{pr}_i:  H \oplus H =H' \oplus H'
\to W_i=W_i'$
as the orthogonal projection.
Then, the composition
$$F_i \equiv \text{pr}_i \circ F: W_i' \to W_i$$
gives a strongly finitely approximable data with some $s_i,r_i$.

\vspace{3mm}

\begin{prop}
There is a $\mathbb {Z}_2$ equivariant $*$-isomorphism
$$K_1^{\mathbb{Z}_2}(S{\frak C}_F(H \oplus H) ) 
\cong K_1^{\mathbb{Z}_2}(S{\frak C}(H \oplus H)).$$
\end{prop}

\begin{proof}
{\bf Step 1:}
By the same argument as the toy case, $F$ is metrically proper, 
and it is $\mathbb{Z}_2$-equivariantly properly homotopic to the involution 
$I: H \oplus H \cong H \oplus H$  by $F^t$.

\begin{sublem}\label{invo}
There is a $\mathbb {Z}_2$-equivariant $*$-isomorphism
$$S{\frak C}_I(H \oplus H) \cong S{\frak C}(H \oplus H).$$
\end{sublem}

\begin{proof}
By construction, 
$$S{\frak C}_I(H \oplus H) = \{  \tilde{u} ; u \in S{\frak C}(H \oplus H) \}$$
where $\tilde{u}(a,b) = I^*(u(a,b))$ with $a,b \in H$.
\end{proof}

{\bf Step 2:}
It follows from Lemma \ref{comp} that 
$F_i^{-1}(D_{s_i} \cap W_i) \subset D_{r_i}' \cap W_i'$ holds.
As in the toy case, one may assume the same property
$$(F_i^t)^{-1}(D_{s_i} \cap W_i) \subset D_{r_i} \cap W_i' \equiv D_{r_i}'$$
where $F_i^t = \text{pr}_i \circ F^t$.

$K$-theory is stable under these continuous deformations
so that  the isomorphism holds
$$K_1^{{\mathbb Z}_2}(S_{r_i}{\frak C}_{F_i^0}(D'_{r_i})) \cong 
K_1^{{\mathbb Z}_2}(S_{r_i}{\frak C}_{F_i^1}(D'_{r_i})) .$$

{\bf Step 3:}
Recall the induced Clifford $C^*$-algebra
$S{\frak C}_F(H)$ whose  element $\{\alpha_i\}_i$
satisfies the equality
$\alpha_{i} = F_{i}^* \beta (u_{i_0})$
for some $u_i = \beta(u_{i_0} )  \in S_{s_i}{\frak C}(D_{s_i} \cap W_i)$
and all $i \geq i_0$. Here, 
$S_{r_i}{\frak C}_{F_i}(D'_{r_i}) $ is defined 
as the image of $F_i^* : S_{s_i}{\frak C}(D_{s_i}\cap W_i )  \to
S_{r_i}{\frak C}_{l_i}(D'_{r_i})  = S_{r_i}{\frak C}(D'_{r_i})$
(and $l_i$ is the identity in this particular case).

Note that $F_i|D_{r_i}'$ has non-zero degree.
We claim that
there is a $*$-homomorphism
$$\Phi_i: S_{r_i}{\frak C}_{F_i}(D'_{r_i}) 
\to S_{r_{i+1}}{\frak C}_{F_{i+1}}(D'_{r_{i+1}})$$
which sends $\alpha_i$ to $\alpha_{i+1}$.
In fact $\alpha_i$ uniquely determines $u_i$.
Suppose the contrary, and choose two elements  $u_i,u_i' 
\in S_{s_i}{\frak C}(D_{s_i}\cap W_i)$ with 
$F_i^*(u_i)=F_i^*(u_i')$.
If $u_i \ne u_i'$ could hold, then there exists $m \in D_{s_i}\cap W_i$
with $u_i(m) \ne u_i'(m)$. However, since $F_i$ has non-zero degree
and is  hence surjective, there exists $x \in D_{r_i}$ with $F_i(x)=m$.
Then, we have the equality
$u_i(m)= F_i^*(u_i)(x) =  F_i^*(u'_i)(x) =u_i'(m)$, which 
 contradicts to the assumption.

Now, since 
$F_i^*: S_{s_i}{\frak C}(D_{s_i}\cap W_i) \to 
S_{r_i}{\frak C}_{F_i}(D'_{r_i})  \subset 
S_{r_i}{\frak C}_{l_i}(D'_{r_i}) $ 
is an isometric  $*$-embedding, 
it follows that the inverse
$$(F_i^*)^{-1}: S_{r_i}{\frak C}_{F_i}(D'_{r_i}) 
\to S_{s_i}{\frak C}(D_{s_i}\cap W_i) $$ is $*$-isomorphic.
Then,
$\Phi_i $
is given by the compositions
$ F_{i+1}^* \circ \beta \circ (F_i^*)^{-1}$.


{\bf Step 4:}
Then, 
\begin{align*}
& K_1^{{\mathbb Z}_2}((S{\frak C}_F(H \oplus H) )  \cong
 \lim_i \ K_1^{{\mathbb Z}_2}(S_{r_i}{\frak C}_{F_i^1}(D_{r_i})) \\
& \cong 
 \lim_i \ K_1^{{\mathbb Z}_2}(S_{r_i}{\frak C}_{F_i^0}(D_{r_i}) )
 \cong K_1^{\mathbb{Z}_2}(S{\frak C}_I(H \oplus H)).
\end{align*}
By  Sublemma \ref{invo}, 
we have the desired isomorphism.
\end{proof}

\subsubsection{Finite cyclic case}
Consider the space
$$S^1_l = \mathbb{R} / l \mathbb{Z}= [0,l]/ \{0 \sim l\}$$
and
$W^{k,2}(S^1_l)$ which  is generated by 
$\sin( 2\pi \frac{k}{l}   s)$ and $\cos( 2\pi \frac{k}{l}   s)$
for $k \in \mathbb{Z}$.

 Consider the Sobolev spaces
  $$W^{k,2}(S^1_l)_0 ,  W^{k,2}(S^1_l)_1,
  \cdots ,W^{k,2}(S^1_l)_{l-1}
   \ \subset \ W^{k,2}(S^1_l)$$
   which  are, respectively,  generated by $W^{k,2}(i,i+1)_0$. Then, 
  $W^{k,2}(S^1_l)_i $ is naturally isometric to $W^{k,2}(S^1_l)_{i+1} $, by the shift 
  $T: u_i \to u_{i+1}$ by $u_{i+1}(s) := u_i(s-1)$, of order $l$.
      Thus,  we can identify these Hilbert spaces 
      by the same symbol $H$
  and, so  the inclusion 
  $H \oplus H  \oplus \dots \oplus H \subset W^{k,2}(S^1_l)$ holds.

   We again consider the non linear map
 \begin{align*}
 & F: H^l = H \oplus H \oplus \dots \oplus H 
  \to H \oplus H \oplus \dots \oplus H , \\
 & F(a_1, \dots,a_l) =( a_1+T(a_l)^3, a_2 +T(a_1)^3, \dots,
 a_l +T(a_{l-1})^3).
 \end{align*}
Then,  the map can be written as
$$\begin{pmatrix}
a_1 \\
a_2 \\
\dots \\
a_l
\end{pmatrix}
\to
\begin{pmatrix}
a_1 +a_l^3 \\
a_2+a_1^3 \\
\dots \\
a_l + a_{l-1}^3
\end{pmatrix}$$
By the same argument as the toy case, this is metrically proper, 
and its nonlinear part is compact on each bounded set.

By use of $F^t$ as above, 
$F$ is $\mathbb{Z}_l$-equivariantly properly homotopic to the 
cyclic shift  $T$.
By a similar argument, we have the following corollary.

\begin{cor}
There is a $\mathbb {Z}_l$-equivariant $*$-isomorphism
$$K_1^{\mathbb{Z}_l}(S{\frak C}_F(H^l ) ) \cong K_1^{\mathbb{Z}_l}(S{\frak C}(H^l))
\cong K_1^{\mathbb{Z}_l}(S)
$$
where ${\mathbb Z}_l$ acts on $H^l$ by the cyclic permutation
of the components. 
\end{cor}

The above computation is  applicable
to more general situations of $F$,  and is not restricted to 
such a specified form of the non linear term.

\subsubsection{Infinite cyclic case}
It is not so immediate to extend the above finite cyclic case  
  to the infinite case, following the same approach.
For example the map
$l^2({\mathbb Z}) \to l^2({\mathbb Z})$ 
by $\{a_i\}_i \to \{a_{i+1}^3\}_i$ is not proper.

Therefore,  we use a very specific approach  to compute the
${\mathbb Z}$ case.
Let $H$ be the Hilbert space identified 
as $H = W^{k,2}(0,1)_0 \subset W^{k,2}({\mathbb R})$,
 and 
let ${\bf H}$ be the closure of the sum 
$\oplus_{i \in {\mathbb Z}} \ H_i$, 
where $H_i$ are the copies of the same $H$.
Then, the  
 $T$ orbit of $H$, 
$\{T^n(H)\}_{n \in {\mathbb Z}}$ generates 
${\bf H} \subset W^{k,2}({\mathbb R})$,
where $T: W^{k,2}(i,i+1)_0 \cong W^{k,2}(i+1,i+2)_0$ 
is the shift as before.

Consider the map $F: {\bf H} \to {\bf H}$ by
$$F: 
\begin{pmatrix}
\dots \\
a_{-1} \\
a_0 \\
a_1 \\
\dots \\
\end{pmatrix}
\to
\begin{pmatrix}
\dots \\
a_{-1} +a_{-2}^3 \\
a_0+a_{-1}^3 \\
a_1+ a_0^3 \\
\dots \\
\end{pmatrix}.$$
Let $H'_l $  be spanned by the vectors
$(a_{-l}, \dots, a_l)$.
As in the toy model case, consider  the approximation
$F^l : H'_l \to H_l$ by shifting the last component
$$F^l: 
\begin{pmatrix}
a_{-l} \\
a_{-l+1} \\
\dots \\
a_l 
\end{pmatrix}
\to
\begin{pmatrix}
a_{-l}  +a_l^3\\
a_{-l+1}+a_{-l}^3 \\
\dots \\
a_l+ a_{l-1}^3
\end{pmatrix}.$$
There is a finite-dimensional
linear subspace $W_i' \subset H_i'$ with $r_i,s_i >0$
such that   $(F^i, W_i', D'_{r_i})$ gives a ${\mathbb Z}$-finitely approximable data with $l_i  =$ id.

An element $u \in K_0(S{\frak C}_F({\bf H}) \rtimes {\mathbb Z})$
has a representative as 
 $u =[p]- [\pi(p)] $, where
 $$p \in Mat( \overline{S{\frak C}_F({\bf H}) \rtimes {\mathbb Z})}.$$
 Here, $p$  can be approximated as
 $$p' \in Mat(C(\{-l, \dots, l\}, \overline{S{\frak C}_{F^l}(W_l')}) ).$$
 
 The rest of the process is parallel to the toy model case, 
and so one can proceed in the same way,
and then obtain   an isomorphism
 \begin{align*}
 K_0(S{\frak C}_F({\bf H}) \rtimes {\mathbb Z})
&  \cong K_0(S{\frak C}({\bf H}) \rtimes {\mathbb Z}) \\
 & \cong K_1(C^*({\mathbb Z})) = K^1(S^1) \cong {\mathbb Z}.
 \end{align*}

\vspace{1.5cm}

Tsuyoshi Kato

Department of Mathematics,

Faculty of Science,

Kyoto University,

Kyoto 606-8502,
Japan.

e-mail: tkato@math.kyoto-u.ac.jp

\end{document}